\documentclass[11pt]{amsart}
\usepackage{amssymb,amsmath,txfonts,mathrsfs,mathtools,titletoc, tikz, stmaryrd, bm}
\usepackage{color}
\usepackage{graphicx}
\usepackage{float}
\usepackage{esint}
\usepackage{hyperref}

\usepackage[alphabetic]{amsrefs}
\newtheorem{theorem}{Theorem}[section]
\newtheorem{prop}[theorem]{Proposition}
\newtheorem{lemma}[theorem]{Lemma}
\newtheorem{remark}[theorem]{Remark}

\newtheorem{question}[theorem]{Question}
\newtheorem{definition}[theorem]{Definition}

\newtheorem{cor}[theorem]{Corollary}
\newtheorem{example}[theorem]{Example}
\newtheorem{conj}[theorem]{Conjecture}

\numberwithin{equation}{section}
\allowdisplaybreaks[4]

\usepackage[T1]{fontenc}

\def\pf{{\it Proof:}~}
\def\R{{\bf\mathbb R} }
\begin{document}
	
\title{Integrals and Rigidity on Manifolds with Nonnegative Ricci Curvature}
\author{Zixuan Chen, Guoyi Xu, Shuai Zhang}
\address{Zixuan Chen\\ Department of Mathematical Sciences\\Tsinghua University, Beijing\\P. R. China}
\email{chenzixu25@mails.tsinghua.edu.cn}
\address{Guoyi Xu\\ Department of Mathematical Sciences\\Tsinghua University, Beijing\\P. R. China}
\email{guoyixu@tsinghua.edu.cn}
\address{Shuai Zhang\\ Department of Mathematical Sciences\\Tsinghua University, Beijing\\P. R. China}
\email{zhangshu22@mails.tsinghua.edu.cn}
\date{\today}

\begin{abstract}
We prove the general sharp mean value inequality for non-negative superharmonic functions and its corresponding rigidity, which removes the radius restriction of Schoen-Yau's classical result about this inequality. And we obtain an explicit formula of the asymptotic scaling invariant integral of weighted scalar curvature, on three dimensional complete Riemannian manifolds with non-negative Ricci curvature and maximal volume growth. As an application, we use this formula to give another proof of Hamilton's pinching conjecture in this case.
\\[3mm]
\end{abstract}
\thanks{G. Xu was partially supported by NSFC 12141103.}

\maketitle

\titlecontents{section}[0em]{}{\hspace{.5em}}{}{\titlerule*[1pc]{.}\contentspage}
\titlecontents{subsection}[1.5em]{}{\hspace{.5em}}{}{\titlerule*[1pc]{.}\contentspage}
\tableofcontents

\section{Introduction}

The interplay between curvature, topology, and analysis on Riemannian manifolds forms a central theme in geometric analysis. A particularly rich line of inquiry investigates how non-negative Ricci curvature, govern the behavior of solutions to elliptic and parabolic partial differential equations, and conversely, how the properties of such solutions reveal deep geometric and topological information about the underlying space. 

Two fundamental tools in this interplay are the mean value formula for harmonic functions and mean value inequalities for subharmonic functions, which provide pointwise control by averages and are deeply connected to the underlying geometry. 

The classical mean value formula characterizes harmonic functions on Euclidean space. We can not expect the mean value formula to hold for general harmonic functions on Riemannian manifolds. Instead, after introducing suitable weight function $b(x) = G(q,x)^{1/(2-n)}$ (where $G(q,x)$ is the minimal positive Green function), which serves as an "analytic" substitute for the distance function; we can show the similar weighted mean value formula holds. 

On general Riemannian manifolds, a mean value inequality often holds for subharmonic functions, but the optimal constant and the geometric implications of the case of equality are subtle questions. A classical result by Schoen and Yau (\cite[Proposition $2.1$, Chapter $1$]{SY}) established a sharp mean value inequality for non-negative superharmonic functions on manifolds with non-negative Ricci curvature. However, their result required the radius of geodesic ball less than the injectivity radius at the center of the geodesic ball. 

One result of this paper is the sharp mean value inequality without radius restriction as follows.

\begin{theorem}[Sharp Mean Value Inequality and Rigidity]\label{thm main-1}
{Let $(M^n, g)$ be a complete Riemannian manifold with $\operatorname{Ric} \ge 0$. Suppose $f\in C^\infty(M^n)- \{0\}$ satisfies $f\geq 0$ and $\Delta \leq 0$. Then 
\[
f(x) \ge \frac{1}{n\omega_n r^{n-1}} \int_{\partial B_x(r)} f \, d\sigma, \quad \quad \forall x\in M^n, r> 0.
\]
Furthermore, if equality holds for some $x$ and $r>0$, then the geodesic ball $B_x(r)$ is isometric to the Euclidean ball $B^n(r) \subseteq \mathbb{R}^n$.
}
\end{theorem}

\begin{remark}\label{rem main novel pt in Main-1}
This result generalizes the classical Schoen--Yau inequality by removing the injectivity radius constraint. Our approach considers the pullback of a superharmonic function via the exponential map and establishes a generalized divergence theorem (Lemma \ref{lem generalized divergence theorem for geodesic ball}) for geodesic balls. This lemma, which handles the pullback function on the tangent space, is of independent interest and may have other applications in the future study.

Another new tool is the quantity \(F(t)\) introduced in Lemma \ref{lem decreasing of average integral}, which coincides with the one studied by Schoen and Yau when the radius is less than the injectivity radius. The key observation is that \(F(t)\) is non-increasing for all \(t>0\). Instead of showing the upper derivative of $F$ is non-positive; we directly estimate the difference \(F(t)-F(t-s)\). The argument use a nested interval argument and the generalized divergence theorem to derive a contradiction, if \(F\) were to increase. The corresponding rigidity theorem, which asserts that equality in the mean value inequality forces the geodesic ball to be Euclidean, is a geometric consequence that underscores the sharpness of the constant.
\end{remark}

Parallel to the above monotonicity argument, another theme in geometric analysis is understanding the asymptotic structure of complete non-compact manifolds with non-negative Ricci curvature using monotonicity. Colding \cite{Colding} and Colding-Minicozzi \cite{CM, CM-AJM} developed powerful monotonicity formulas for harmonic functions, which are crucial for analyzing the Green's function and the asymptotic volume ratio $\mathrm{V}_M$. These tools link the coarse geometric invariant like $\mathrm{V}_M$ to the integral of scalar curvature.

The study of integral of the curvature started from the well-known Gauss-Bonnet Theorem: for any compact $2$-dim Riemannian manifold $(M^2, g)$, 
\begin{align}
\int_{M^2} K d\mu= 2\pi\chi(M^2), \nonumber 
\end{align}
where $K$ is the Gaussian curvature of $M^2$, $d\mu$ is the element of area of $M^2$, $\chi(M^2)$ is the Euler characteristic of $M^2$. Cohn-Vossen \cite{CV} studied the integral of the curvature on complete $2$-dim Riemannian manifold, obtained the so-called Cohn-Vossen's inequality: If $(M^2, g)$ is a finitely connected, complete, oriented Riemannian manifold, and assume $\int_{M^2} Kd\mu$ exists as extended real number, then $\int_{M^2} Kd\mu\leq 2\pi\chi(M^2)$. Also see \cite{Huber} and \cite{Hartman} for related works on $2$-dim Riemannian manifolds.

Motivated to get a generalization of the Cohn-Vossen's inequality, Yau \cite[Problem $9$]{Yau} posed the following question: 
\begin{question}[Yau]\label{ques Yau-1}
{For any complete Riemannian manifold $(M^n, g)$ with $Rc\geq 0$, any $p\in M^n$, is it true that $\lim\limits_{r\rightarrow \infty} r^{2- n}\int_{B_r(p)} R< \infty$ ?
}
\end{question}

There are several progress on this question, see \cite{Xu-IOC1}, \cite{Zhu}, \cite{Zhu-1}, \cite{Xu} and \cite{MW}. We present the following example, which explains the completeness of the whole Riemannian manifold is crucial to Yau's conjecture.

\begin{example}
{We consider a sequence of 2-dimensional Riemannian manifolds $(M^2, g_i)$ with points $p_i$ such that the geodesic balls $B_{p_i}(1)$ have sectional curvature $K \geq 0$, and
\[
\lim_{i \to \infty} \int_{B_{p_i}(1)} R  dV = \infty,
\]
where $R = 2K$ is the scalar curvature.

Using geodesic polar coordinates centered at $p_i$, the metric takes the form:
\[
ds^2 = dr^2 + f_i(r)^2 d\theta^2, \quad \quad  f_i(r)=r-\int_0^r\frac{t^2}{\sqrt{1-t^2+1/i^2}}\mathrm{d}t.
\]
with $f_i(0) = 0$, $f_i'(0) = 1$, and all even-order derivatives vanishing at $r = 0$ to ensure smoothness.

Direct computation yields:
\[
\int_{B_{p_i}(1)} R  dV = 2 \int_0^{2\pi} \int_0^1 K f_i(r)  dr  d\theta = 4\pi \int_0^1 (-f_i''(r))  dr= 4\pi (1 - f_i'(1))\rightarrow \infty.
\]
}
\end{example}

Our second set of results concerns the asymptotic integral geometry of three-dimensional manifolds with non-negative Ricci curvature and maximal volume growth ($\mathrm{V}_M > 0$). By analyzing the asymptotic properties of the function $b(x)$ and its level sets, we establish an explicit, scale-invariant integral formula in Theorem \ref{thm main-2} for a weighted scalar curvature. 

\begin{theorem}\label{thm main-2}
{For a complete non-compact Riemannian manifold $M^3$ with $Rc\geq 0$ and maximum volume growth, there is
\begin{align}
\lim_{r\rightarrow \infty}\frac{\int_{b\leq r}R\cdot |\nabla b|}{r}=8\pi \big[1- \mathrm{V}_{M}\big], \label{integral fomula in main-2},
\end{align}
where $R$ is the scalar curvature.
}
\end{theorem}

\begin{remark}
{In \cite{Xu-IOC1}, it is proved: for a complete non-compact Riemannian manifold $M^3$, which is non-parabolic with $Rc\geq 0$, we have 
\begin{align}
\varlimsup_{r\rightarrow \infty}\frac{\int_{b\leq r}R\cdot |\nabla b|}{r}\leq 8\pi \big[1- \mathrm{V}_{M}\big]. \label{ineq for non-parabolic}
\end{align}
Theorem \ref{thm main-2} can be viewed as a refinement of (\ref{ineq for non-parabolic}) in the maximal volume growth case.

The identity (\ref{integral fomula in main-2}) provides a precise quantitative measurement of how the geometry deviates from the Euclidean standard and links the asymptotic scalar curvature distribution directly to the asymptotic volume ratio.
}
\end{remark}

It is illuminating to explain why (\ref{integral fomula in main-2}) holds from an intuitive view. Consider a complete non-compact three-manifold $(M^3,g)$ with $\operatorname{Ric} \ge 0$ and maximal volume growth ($\mathrm{V}_M>0$). By the theory of Cheeger--Colding \cite{CC-Ann}, after rescaling the metric by $r^{-2}$, the pointed sequence $(M, r^{-2}g, q)$ converges in the pointed Gromov--Hausdorff sense to a metric cone $C(X)$, where $X$ is a compact metric space (possibly with singularities). 

If the convergence is smooth outside the vertex, then $X$ is a smooth surface with sectional curvature $K_X\ge 1$ and the cone metric satisfies
\begin{align}
g_{C(X)} = dr^{2} + r^{2} g_{X}, \quad  R_{C(X)} = \frac{2(K_X-1)}{r^2}, \qquad r>0. \nonumber 
\end{align}

Note we have 
\begin{align}
\mathrm{V}_M= \mathrm{V}_{C(X)}= \lim_{t\to\infty} \frac{\operatorname{Vol}(B(t))}{\omega_{3} t^{3}}= \lim_{t\to\infty} \frac{\int_{0}^{t}dr \int_{X} r^{2}}{\omega_{3} t^{3}}= \frac{\mathrm{Vol}(X)}{4\pi}, \label{V(X) and VM} 
\end{align}
where $B(t)$ is the ball in $C(X)$

Using Gauss--Bonnet Theorem and (\ref{V(X) and VM}), we get 
\begin{align}
\int_{B(1)} R \, dV 
   &= \int_{0}^{1}dr \int_{X} \frac{2(K_{X}-1)}{r^{2}} \, r^{2} 
   = 2\left( [\int_{X} K_{X} ] - \mathrm{Vol}(X) \right)\nonumber \\
   &=8\pi\,(1-\mathrm{V}_{M}). \label{cone curv formula}
\end{align}

From \cite{CC-Ann},  for $B_q(r)\subseteq (M^3, g)$ with $Rc\geq 0$ and maximal volume growth, we know that $(B_q(1), r^{-2}g)$ converges to $(B(1), g_{C(X)})$ in Gromov-Hausdorff sense (although $X$ possibly is not unique, but $V_{C(X)}= V_{M^3}$). 

Note we have 
\begin{align}
\lim_{r\rightarrow\infty}\frac{\int_{B_q(r)}R(g)dg}{r}= \lim_{r\rightarrow\infty}\int_{(B_q(1), r^{-2}g)}R(r^{-2}g)d(r^{-2}g). \nonumber 
\end{align}

Although it is not rigorous, from the Gromov-Hausdorff convergence of $(B_q(1), r^{-2}g)$ to $(B(1), g_{C(X)})$,  it is not too wild to expect that the following holds:
\begin{align}
\lim_{r\rightarrow\infty}\int_{(B_q(1), r^{-2}g)}R(r^{-2}g)d(r^{-2}g)= \int_{(B(1), g_{C(X)})}R(C(X))d(g_{C(X)}) . \label{key convergence}
\end{align}

If (\ref{key convergence}) holds, from (\ref{cone curv formula}) we will have 
\begin{align}
\lim_{r\rightarrow\infty}\frac{\int_{B_q(r)}R(g)dg}{r}= 8\pi\,(1-\mathrm{V}_{M^3}). \nonumber  
\end{align}

Based on the above discussion, we make the following conjecture.
\begin{conj}\label{conj integral curv formula}
{For a complete non-compact Riemannian manifold $M^3$ with $Rc\geq 0$ and maximum volume growth, there is
\begin{align}
\lim_{r\rightarrow \infty}\frac{\int_{B_q(r)}R}{r}=8\pi \big[1- \mathrm{V}_{M}\big]. \label{integr curv equa for maximal case}
\end{align}
}
\end{conj}

\begin{remark}\label{rem conj for pole case}
{The conjecture \ref{conj integral curv formula} was proved for manifolds with pole in \cite{Xu}. And Theorem \ref{thm main-2} can be viewed as a sibling version of Conjecture \ref{conj integral curv formula}
}
\end{remark}

Hamilton has proposed the following conjecture:
\begin{conj}\label{conj Hamilton}
{If $(M^3, g)$ is a complete, non-compact Riemannian manifold and Ricci pinched, then $(M^3, g)$ is flat.
}
\end{conj}
The conjecture has been verified with additional assumption by \cite{CZ} and \cite{Lott}. For recent progress on Conjecture \ref{conj Hamilton}, see \cite{DSS} and \cite{LT} for the proof using Ricci flow (also see \cite{HK} for argument by the inverse mean curvature flow).

As an application of Theorem \ref{thm main-2}, we provide an alternative proof of a special case of Hamilton's pinching conjecture in dimension three.

\begin{cor}\label{cor harmilton}
{Assume $(M^3, g)$ is a complete Riemannian manifold with $\operatorname{Ric} \ge \epsilon\cdot R\cdot g \ge 0$ for some fixed $\epsilon > 0$, and is maximum volume growth. Then $M^3$ is isometric to $\mathbb{R}^3$.
}
\end{cor}

\begin{remark}\label{rem compare with Bonnet-Myers proof}
{The Bonnet-Myers Theorem is proved by integrating curvature along geodesic segments. In a similar spirit, Hamilton's pinching conjecture can be seen as a scaling-invariant analogue; the proof of Corollary \ref{cor harmilton} extends this idea by integrating curvature over a region of $M^3$.
}
\end{remark}

Now we explain some innovations of the proof of Theorem \ref{thm main-2}. Deriving the exact limit in (\ref{integral fomula in main-2}), requires knowledge of the topological type of the level sets of $b$, as this allows us to apply the Gauss--Bonnet theorem. In this paper, we study the set (denoted as $\mathcal{T}\cup \mathcal{B}$) of parameters $t \in \mathbb{R}$ for which the level set $b^{-1}(t)$ has genus at least one. Firstly, under non-parabolic assumption, we show 
\begin{align}
\lim_{r\rightarrow\infty}\frac{\int_{[\mathcal{T}\cup \mathcal{B}] \cap (0, r)}dt\int_{b^{-1}(t)}R(b^{-1}(t))}{r}= 0. \nonumber 
\end{align}
Furthermore, under maximal volume growth assumption, we prove that the \emph{asymptotic density} of this set vanishes: that is
\begin{align}
\frac{L^1\bigl([\mathcal{T}\cup \mathcal{B}] \cap (0, t)\bigr)}{t}, 
\end{align}
tends to $0$ as $t \to \infty$, even though $L^1\bigl([\mathcal{T}\cup \mathcal{B}] \cap (0, t)\bigr)$ itself possibly grow (for instance, like $\sqrt{t}$). In other words, under the condition of maximal volume growth, we show that asymptotically almost every level set is topologically a sphere in the above sense. This is a result with independent interest. 

A detailed asymptotic analysis of the integral of the mean curvature of such level sets is essential for establishing the above asymptotic result. For technical reasons, we therefore begin with the weighted volume of the level sets and establish the precise asymptotic results required in the proof. Although portions of these asymptotic results were first obtained by Colding--Minicozzi (\cite{CM}) and one of the authors (\cite{Xu-IOC1}), the argument presented here is self-contained and emphasizes the second derivative of the weighted volume. This focus not only completes the derivation in a unified way but is also of independent interest, paralleling similar techniques used in monotonicity formulas for minimal surfaces and harmonic maps.

We expect that (\ref{integral fomula in main-2}) should in fact hold in the broader \emph{non-parabolic} setting, although the case of non-parabolic manifolds without maximal volume growth remains open. Note there are example manifold which is non-parabolic without maximal volume growth.

The application of the sharp integral curvature formula to Hamilton's conjecture, was established by one of the authors (\cite{Xu}) under the pole assumption. This approach is also related to the method of Huisken and Koeber using inverse mean curvature flow. 

This paper is structured as follows. In Section $2$, we establish foundational mean value formulas using the Green's function. The generalized divergence theorem for geodesic balls is proved in Section $3$. Section $4$ is devoted to the proof of the sharp mean value inequality and its rigidity. Sections $5$ and $6$ develop the asymptotic analysis of the function $b$ and its level sets, culminating in the proof of Theorem \ref{thm main-2} in Section $6$. 

In the rest of paper, unless otherwise mentioned, we always assume $(M^n, g)$ is a non-parabolic manifold  with $n\geq 3$ and $Rc\geq 0$.

\section*{Part I. Integral for harmonic functions}

\section{Mean value formula}

On non-parabolic manifold $M^n$ with $n\geq 3$, there exists a unique minimal positive Green function, denoted as $G(q, x)$, where $q$ is a fixed point on the manifold. To emphasize the base point, we use the notation that
\begin{align*}
    b_q(x)= G(q, x)^{\frac{1}{2- n}}.
\end{align*}

From $\Delta (b_q^{2- n})= 0$, we have 
\begin{equation}
\Delta b_q = (n-1)\frac{|\nabla b_q|^2}{b_q} .\label{equation of Laplace of v}
\end{equation}

From the behavior of Green's function $G(q, x)$ near the singular point $q$, we get 
\begin{align}
\lim_{r\to 0} \Big\{\sup_{x\in B_q(r)} \left| \frac{b_q}{\rho}(x) - 1 \right| + \sup_{x\in B_q(r) }\left| |\nabla b_q(x)| - 1 \right| = 0\Big\},\label{est of gradient of b at 0}
\end{align}
where $\rho(x) = d(x,q)$.

The following result computes the value of scaling invariant weighted volume around the singularity.
\begin{lemma}\label{lem volume near singularity}
{Let $\omega_n$ be the volume of unit ball in $\mathbb{R}^n$, we have
\begin{align*}
\lim_{{r\to 0}} r^{-n} \int_{b_q\le r} |\nabla b_q|^2 = \omega_n.
\end{align*}
}
\end{lemma}

\pf
{From (\ref{est of gradient of b at 0}), for any \( \epsilon \in (0, 1) \), there exists \( \delta > 0 \), such that
\begin{align}
\left||\nabla b_q(x)| - 1\right| < \epsilon, \quad |b_q - \rho|(x) < \epsilon\rho(x), \quad \forall x \in B_q(\delta). \label{inq between rho and b}
\end{align}

For $\delta>0$, we claim that there exists $r_0$, such that if $b_q(x)\le r_0$, then $\rho(x)\le \delta$. Otherwise, there exists sequence $\{q_i\}_{i=1}^\infty$ on $M$, such that 
\begin{align}
\lim_{i\rightarrow\infty}b_q(q_i)= 0, \quad \quad \rho(q_i)> \delta. \nonumber 
\end{align}
This implies $\displaystyle \lim_{i\to \infty}G(q,q_i)\to \infty$, on the other hand $q_i\notin B_q(\delta)$; which contradicts the boundedness of Green function $G(q, \cdot)$ on $M^n- B_q(\delta)$. 

For any \( r < r_0 \), from (\ref{inq between rho and b}) we get
\begin{align}
\{x: \rho(x)\leq\frac{r}{1+ \epsilon}\}\subseteq \{x: b_q(x)\leq r\}\subseteq \{x: \rho(x)\leq \frac{r}{1- \epsilon}\}. \nonumber 
\end{align} 

Therefore we have
\begin{align*}
r^{-n} \int_{\rho\leq \frac{r}{1+\epsilon}} (1-\epsilon)^2 \leq r^{-n} \int_{b_q \leq r} |\nabla b_q|^2 \leq r^{-n} \int_{\rho \leq \frac{r}{1- \epsilon}} (1+\epsilon)^2.
\end{align*}

Hence let $r\to0$, 
\begin{align*}
    \frac{\omega_n(1-\epsilon)^2}{(1+\epsilon)^n}\le \lim_{{r\to 0}} r^{-n} \int_{b_q\le r} |\nabla b_q|^2 \le \frac{\omega_n}{(1- \epsilon)^n}(1+\epsilon)^2.
    \end{align*}
    
Letting $\epsilon\to 0$ in the above, the conclusion follows.
}
\qed

In \cite[Lemma 5.1]{Xu-Growth}, a mean value formula for harmonic functions is established. The following lemma extends this result, characterizing both harmonic and subharmonic functions via their mean value properties.

\begin{lemma}\label{lem mean value formula of harmonic func}
    Let $u\in C^\infty(M^n)$. For any $q\in M^n$ and $r>0$, let 
    \[
        I(u,q,r) = \frac{1}{n\omega_n r^{n-1}}\int_{b_q=r} u|\nabla b_q|.
    \]
    Then the following hold:
    \begin{enumerate}
        \item[(i)] $\Delta u = 0$ if and only if $u(q) = I(u,q,r)$ for all $q, r$.
        \item[(ii)] $\Delta u \ge 0$ if and only if $u(q) \le I(u,q,r)$ for all $q, r$.
    \end{enumerate}
\end{lemma}

\pf
{\textbf{Step (1)}. Firstly, we note $\frac{\partial}{\partial r}= \frac{\nabla b_q}{|\nabla b_q|^2}$, assume the volume element of $b_q^{-1}(r)$ is $J(b_q^{-1}(r))$. Then as in the argument of Lemma \ref{lem derivative of Vp}, we can get $I(u,q,r)$ is differentiable in $r$ and 
\begin{align}
&\quad \quad \frac{d}{dr}\Big(r^{1- n}\int_{b_q= r}u|\nabla b_q|\Big) = (1- n)r^{-n}\int_{b_q= r}u|\nabla b_q|+ r^{1- n}\int_{b_q= r} \langle \nabla u, \frac{\nabla b_q}{|\nabla b_q|}\rangle \nonumber \\
&\quad \quad + r^{1- n}\int_{b_q= r} u\Big(\frac{\nabla_{\nabla b_q}|\nabla b_q|}{|\nabla b_q|^2}+ \frac{\nabla_{\nabla b_q}J(b_q^{-1}(r))}{|\nabla b_q| J(b_q^{-1}(r))}\Big)dx \nonumber \\
&= r^{1- n}\int_{b_q\leq r}\Delta u+ (1- n)r^{-n}\int_{b_q= r}u|\nabla b_q|+ r^{1- n}\int_{b_q= r} \frac{u}{|\nabla b_q|}\Delta b_q, \label{before simplification}
\end{align}
where we use the Divergence Theorem and the fact $\Big(\frac{\nabla_{\nabla b_q}|\nabla b_q|}{|\nabla b_q|}+ \frac{\nabla_{\nabla b_q}J(b_q^{-1}(r))}{J(b_q^{-1}(r))}\Big)= \Delta b_q$.

From (\ref{equation of Laplace of v}) and (\ref{before simplification}), we get
\begin{align}
\frac{d}{dr}\Big(r^{1- n}\int_{b_q= r}u|\nabla b_q|\Big) = r^{1- n}\int_{b_q\leq r}\Delta u. \label{deriv of integ}
\end{align}

Choose \( u = 1 \) in (\ref{deriv of integ}), by Lemma \ref{lem volume near singularity} we get 
\begin{align}
r^{1-n} \int_{b_q = r} |\nabla b_q|= n\omega_n. \label{slice weighted volume}
\end{align}

By coarea formula and (\ref{slice weighted volume}), we have
\begin{align}
\lim_{r\to 0} r^{1-n} \int_{b_q = r} |\nabla b_q | 
&= \lim_{r\to 0} n  r^{-n} \int_0^r \left( \int_{b_q = r} |\nabla b_q|  \right) dt
= \lim_{r\to 0} n  r^{-n} \int_{b_q \le  r} |\nabla b_q|^2\nonumber \\
&= n\omega_n. \label{volume slice near singularity}
\end{align}

\textbf{Step (2)}. From (\ref{volume slice near singularity}), for any $u$, we have
\begin{align}
\lim_{r\rightarrow 0}r^{1- n}\int_{b_q= r}u|\nabla b_q|= u(q)\lim_{r\rightarrow 0}r^{1- n}\int_{b_q= r}|\nabla b_q|= u(q)n\omega_n . \label{limit at 0}
\end{align}

Now if $\Delta u= 0$, from (\ref{deriv of integ}) and (\ref{limit at 0}) we get
\begin{align}
r^{1- n}\int_{b_q= r}u|\nabla b_q|= \lim_{r\rightarrow 0}r^{1- n}\int_{b_q= r}u|\nabla b_q|= u(q)n\omega_n. \nonumber 
\end{align}

If $\Delta u\geq 0$, from (\ref{deriv of integ}) and (\ref{limit at 0}) we get
\begin{align}
r^{1- n}\int_{b_q= r}u|\nabla b_q|\geq \lim_{r\rightarrow 0}r^{1- n}\int_{b_q= r}u|\nabla b_q|= u(q)n\omega_n. \nonumber 
\end{align}

\textbf{Step (3)}. If $u(q) = I(u,q,r)$ for all $q, r$, we show that $u$ is harmonic.

Denote $A(q,r)=\{x\in M:b_q(x)< r\}$. From (\ref{est of gradient of b at 0}), we can take $r_0$ sufficiently small so that $|\nabla b_q|(x)\neq 0$ for any $x\in A(q, r_0)$ and $A(q, r_0)$ lies in a local coordinate chart of $M$. 

From the standard elliptic theory, there exists a harmonic function $v\in C^\infty(\overline{A(q, r_0)})$ with $v=u$ on $\partial A(q, r_0)$. Then from Step (2), $v$ satisfies the mean value property
\begin{align*}
v(q)= \frac{1}{n\omega_nr^{n- 1}}\int_{b_q= r}v|\nabla b_q|, \quad \quad \forall A(q,r)\subset A(q, r_0). \nonumber
\end{align*}

Now we claim $v=u$. Let $h=v-u$, then $h$ also satisfies the mean value property
\begin{align*}
h(q)= \frac{1}{n\omega_nr^{n- 1}}\int_{b_q= r}h|\nabla b_q|, \quad \quad \forall A(q,r)\subset A(q, r_0). \nonumber
\end{align*}

If $h$ is not identically $0$, then $h$ have a nonzero maximum or minimum in $A(q, r_0)$. 

Suppose without loss of generality that $\displaystyle h(z)=\max_{A(q, r_0)}h>0$. Then for small $r>0$ such that $A(z,r)\subset A(q, r_0)$ and $|\nabla b_z|\neq 0$ on $A(z,r)$, we have
\begin{align*}
    0=h(z)-\frac{1}{n\omega_nr^{n- 1}}\int_{b_z= r}h|\nabla b_z|=\frac{1}{n\omega_nr^{n- 1}}\int_{b_z= r}(h(z)-h)|\nabla b_z|\geq 0.
\end{align*}
Hence $h\equiv h(z)$ on $\partial A(z,r)$. We then have $h\equiv h(z)$ on $A(z,r_1)$ for all small $r_1>0$. 

This implies $h$ is constant on the whole $A(q,r_0)$ by the connectedness of $A(q, r_0)$. 

The constant is $0$ since $h=0$ on $\partial A(q, r_0)$, which contradicts our hypothesis that $h(z)>0$. 

So $h$ is identically zero and $u\equiv v$. Thus $u$ is harmonic near $q$. The conclusion then follows.

\textbf{Step (4)}. Now if $u(q) \le I(u,q,r)$ for all $q, r$, we show $\Delta u \ge 0$.

Suppose, for contradiction, that $\Delta u(q_0) < 0$ at some point $q_0 \in M$. Then by continuity, there exists a small $\epsilon > 0$ and a geodesic ball $B_\delta(q_0)$ such that $\Delta u < -\epsilon < 0$ on $B_\delta(q_0)$. 

Choose $r_0 \in (0, \delta)$ sufficiently small so that $A(q_0, r_0) \subset B_\delta(q_0)$ and $|\nabla b_{q_0}| \ne 0$ on $A(q_0, r_0)$. From Step (1), for all $r \in (0, r_0)$, we have the identity
\begin{align}
\frac{d}{dr}I(u, q_0, r) = r^{1-n} \int_{b_{q_0} \le r} \Delta u< 0. \label{deriv-sub}
\end{align}

We must have
\[
I(u, q_0, r) < \lim_{r \to 0^+} I(u, q_0, r) = n\omega_n u(q_0) \quad \text{for all } r \in (0, r_0).
\]

But this contradicts the hypothesis that $u(q_0) \le I(u, q_0, r)$ for all $r > 0$.

Therefore, our assumption that $\Delta u(q_0) < 0$ must be false, and $\Delta u \ge 0$ everywhere on $M^n$.
}
\qed

\section{Divergence theorem for geodesic balls}

For the exponential map $\exp_q: T_qM^n\rightarrow M^n$, we define $\mathcal{R}:\mathbb{S}^{n-1}\to \mathbb{R}_+\cup\{\infty\}$ as follows:
    \begin{align*}
        \mathcal{R}(\theta):= \sup\{r>0:\ d_g(q, \exp_q((r+ \epsilon)\theta))= r+ \epsilon \ \text{for some}\ \epsilon> 0 \}.
    \end{align*}
    
Define 
    \begin{align}
       & A=\{t\cdot \theta:\theta\in \mathbb{S}^{n-1},0\leq t<\mathcal{R}(\theta)\}, \nonumber \\
       & \Sigma(t)=\{\theta\in\mathbb{S}^{n-1}:\mathcal{R}(\theta)>t\}, \quad \quad        \tilde\Sigma(t)=\{\theta\in\mathbb{S}^{n-1}:\mathcal{R}(\theta)\geq t\}. \nonumber 
    \end{align}
     For fixed $q\in M$, we define the geodesic ball $B_q(t)= \{x\in M^n: d_g(q, x)< t\}$. We have
    \begin{align*}
        &\exp_q(t\cdot\Sigma(t))\subset\exp_q(t\cdot\tilde\Sigma(t))=\partial B_q(t).
    \end{align*}
    
Define 
\begin{align*}
    Z=\{t>0:\mathcal{H}^{n-1}(\tilde\Sigma(t)\backslash \Sigma(t))>0\}.
\end{align*}
Let $\partial_*B_q(t)$ be the measure theoretic boundary defined by
    \begin{align}
        \partial_*B_q(t)=\left\{y\in M:\begin{array}{ll}
           \varlimsup\limits_{r\to 0}\frac{\mathcal{H}^n(B_y(r)\cap B_q(t))}{r^n}>0\ \text{and}\ 
             \varlimsup\limits_{r\to 0}\frac{\mathcal{H}^n(B_y(r)\backslash B_q(t))}{r^n}>0
        \end{array}\right\}.\label{measure theoretic boundary}
    \end{align}

\begin{lemma}\label{lem geodesic ball has locally finite perimeter}
{For $t\in \mathbb{R}^+\backslash Z$, the geodesic ball $B_q(t)$ has locally finite perimeter and 
\begin{align}
&\mathcal{H}^{n- 1}(\partial B_q(t)\backslash \exp_q(t\cdot \Sigma(t)))= 0. \label{vanish diff on sphere}
\end{align}
}
\end{lemma}

\pf
{Firstly we fix $t\in \mathbb{R}^+\backslash Z$ in the rest argument. Note the map $F(v)= \exp_q(t\cdot v): \mathbb{S}^{n- 1}\rightarrow (M^n, g)$ is a Lipschitz map with Lipschitz constant $C> 0$. Now from the area formula, for $t\notin Z$, we get
\begin{align}
&\mathcal{H}^{n- 1}(\partial B_q(t)\backslash \exp_q(t\cdot \Sigma(t)))= \mathcal{H}^{n- 1}\Big(\exp_q\big(t\cdot \tilde{\Sigma}(t)\big)\backslash \exp_q\big(t\cdot \Sigma(t)\big)\Big) \nonumber \\
&\leq C\cdot \mathcal{H}^{n- 1}( \tilde{\Sigma}(t)\backslash \Sigma(t))= 0. \nonumber 
\end{align}

Now we have 
\begin{align}
\mathcal{H}^{n- 1}(\partial_* B_q(t))&\leq \mathcal{H}^{n- 1}(\partial B_q(t)) \nonumber \\
&\leq \mathcal{H}^{n- 1}(\partial B_q(t)\backslash \exp_q(t\cdot \Sigma(t)))+ \mathcal{H}^{n- 1}(\exp_q(t\cdot \Sigma(t)))\nonumber \\
&\leq C\cdot \mathcal{H}^{n- 1}(\Sigma(t))\leq C\cdot \mathcal{H}^{n- 1}(\mathbb{S}^{n- 1})< \infty. \label{finite perimeter}
\end{align}

The conclusion follows from (\ref{finite perimeter}) and \cite[Theorem $5.23$]{EG}. 
}
\qed

Let $\{\theta, v_2, \cdots, v_n\}$ be an orthonormal basis of the tangent space $T_q M^n\simeq \mathbb{R}^n$. Here, $\theta, v_i\in \mathbb{S}^{n-1}$, since the unit tangent sphere at $q\in M^n$ is diffeomorphic to $\mathbb{S}^{n-1}$. And we regard $T_{t\theta}(T_qM^n)\simeq T_qM^n\simeq \mathbb{R}^n$, denote
\begin{align*}
&W_i(t, \theta)=d\exp_q(t\theta)(tv_i), \quad \quad \forall 2\leq i\leq n,  t> 0,\nonumber \\
&J(t, \theta) = \sqrt{\det\left(\left\langle W_i(t, \theta), W_j(t, \theta) \right\rangle_{i,j=2}^n\right)}, \quad \quad \forall t> 0, \theta\in \mathbb{S}^{n- 1}. \nonumber 
\end{align*}

By Gauss Lemma, we have
\begin{align}
  &  \langle d\exp_p(t\theta)(\theta), d\exp_p(t\theta)(\theta) \rangle=1, \label{unit entry} \\
   & \langle d\exp_p(t\theta)(\theta), d\exp_p(t\theta)(v_i) \rangle=0,\quad i=2,\cdots n. \label{perpendicular entry} 
\end{align}

Consider the map $\exp_q: \mathbb{R}^n\to M^n$, then 
\begin{align}
\mathrm{Jac}(\exp_q)(t\cdot \theta)&= \sqrt{\mathrm{det}[(D\exp_q)\cdot (D\exp_q)^*]}\nonumber\\
&= \sqrt{\det\left(\left\{ \langle d\exp_q(t\theta)(v_i), d\exp_q(t\theta)(v_j) \rangle \right\}_{i,j=2}^n\right)}  \nonumber \\
&= \frac{J(t, \theta)}{t^{n- 1}} . \label{Jacobian of exp map}
\end{align}

Assume $z= (z_1, \cdots, z_n)= t\cdot \theta\in \mathbb{R}^n$, we know $dz_1 \cdots dz_n = t^{n-1}dtd\theta$. From the area formula (see \cite[Theorem $3.9$]{EG}), for any $L^n$-measurable function $h: \mathbb{R}^n\rightarrow \mathbb{R}$, we get
\begin{align}
&\int_{\mathbb{R}^n} h(t\cdot \theta)J(t, \theta)dtd\theta= \int_{\mathbb{R}^n} h(t\cdot \theta)\mathrm{Jac}(\exp_q) t^{n- 1}dtd\theta \nonumber \\
&= \int_{\mathbb{R}^n} h(z)\mathrm{Jac}(\exp_q) dz_1\cdots dz_n = \int_{M^n} \Big[\sum_{t\cdot \theta\in \exp_q^{-1}(x)}h(t\cdot \theta)\Big] d\mu(x). \label{chaning variable formula}
\end{align}

Similarly for fixed $t> 0$, consider the map $\tilde{F}(\theta)= \exp_q(t\cdot \theta): \mathbb{S}^{n-1}\rightarrow M^n$, from (\ref{Jacobian of exp map}), we get
\begin{align}
\mathrm{Jac}(\tilde{F})= \mathrm{Jac}(\exp_q)\cdot t^{n-1}= J(t, \theta). \nonumber 
\end{align}

For any $\mathcal{H}^{n-1}$-measurable function $h: \mathbb{S}^{n- 1}\rightarrow \mathbb{R}$, we have
\begin{align}
&\int_{\mathbb{S}^{n- 1}} h(\theta)J(t, \theta)d\theta= \int_{\mathbb{S}^{n- 1}} h(\theta)\mathrm{Jac}(\tilde{F}) d\theta \nonumber \\
&= \int_{\mathbb{S}^{n- 1}(t)} h(\frac{y}{t})\mathrm{Jac}(\exp_q) dy = \int_{M^n} \Big[\sum_{t\cdot \theta\in \exp_q^{-1}(x)}h(\theta)\Big] d\mu(x). \label{chaning variable formula-slice}
\end{align}

\begin{lemma}\label{lem divergence for geodesic ball}
{We have $\mathcal{H}^{1}(Z)=0$ and
\begin{align}
        \int_{B_q(t)}\Delta f= \int_{\exp_q(t\cdot\Sigma(t))}\frac{\partial f}{\partial \vec{n}}\mathrm{d}\mu_{\partial B_q(t)},\quad \quad \quad  t\in (0,\infty)\backslash Z .\label{eq derivative of f negative except null set}
    \end{align}
}
\end{lemma}

\pf
{\textbf{Step (1)}. Note
    \begin{align*}
        \partial A=\{\mathcal{R}(\theta)\theta:\theta\in \mathbb{S}^{n-1}\},
    \end{align*}
then from \cite[Proposition III.3.1]{CI}, we have
\begin{equation*}
    \mathcal{H}^n(\partial A)=\int_{\R^n}\chi_{\partial A}dx=\int_{\mathbb{S}^{n-1}}d\theta \int_{0}^\infty \chi_{\partial A}(t\theta)t^{n-1}dt=0.
\end{equation*}

    From the coarea formula, we have
    \begin{align*}
        0&=\mathcal{H}^n(\partial A)= \int_0^\infty\int_{\{\mathcal{R}(\theta)
        =t\}}t^{n-1}\mathrm{d}\theta\mathrm{d}t=\int_0^\infty\int_{\tilde\Sigma(t)\backslash\Sigma(t)}t^{n-1}\mathrm{d}\theta\mathrm{d}t 
        =\int_0^\infty\mathcal{H}^{n-1}(\tilde\Sigma(t)\backslash\Sigma(t))t^{n-1}\mathrm{d}t.
    \end{align*}
    So we must have $\mathcal{H}^1(Z)=0$.

\textbf{Step (2)}. From \cite[Theorem 5.16]{EG} and Lemma \ref{lem geodesic ball has locally finite perimeter}, the divergence theorem can be applied to $B_q(t)$ to get
    \begin{align}
        \int_{B_q(t)}\Delta f=\int_{\partial_*B_q(t)}\frac{\partial f}{\partial \vec{n}}\mathrm{d}\mu_{\partial B_q(t)}\quad \forall t\in (0,\infty)\backslash Z.  \label{divergence on sphere} 
    \end{align}
    
Since $\exp_q$ is a diffeomorphism on $A$ from \cite[Theorem III.2.2]{CI}, we have $\exp_q(t\cdot\Sigma(t))\subset \partial_* B_q(t)$. Combining (\ref{vanish diff on sphere}), we have $\mathcal{H}^{n- 1}(\partial_*B_q(t)\backslash \exp_q(t\cdot\Sigma(t)))= 0$. Therefore we get 
\begin{align}
\int_{\partial_*B_q(t)}\frac{\partial f}{\partial \vec{n}}\mathrm{d}\mu_{\partial B_q(t)}=\int_{\exp_q(t\cdot\Sigma(t))}\frac{\partial f}{\partial \vec{n}}\mathrm{d}\mu_{\partial B_q(t)},  \quad \forall  t\in (0,\infty)\backslash Z. \label{reduced to smooth part}
\end{align}    

The conclusion follows from (\ref{divergence on sphere}) and (\ref{reduced to smooth part}). 
}
\qed

\begin{definition}\label{def tilde f defined by f}
{For any $f\in C^\infty(M^n)$, define 
\begin{align}
\tilde{f}: \mathbb{R}^+\times \mathbb{S}^{n-1}\rightarrow M^n, \nonumber 
\end{align}
by $\tilde{f}(t, \theta)= f(\exp_q(t\theta))$. Note $\tilde{f}$ is a smooth function.
}
\end{definition}

\begin{lemma}[Generalized divergence theorem for geodesic ball]\label{lem generalized divergence theorem for geodesic ball}
    Suppose that $f\in C^\infty(M^n)$, then
        \begin{align}
        \int_{\tilde\Sigma(t)}\frac{\partial \tilde{f}(t,\theta)}{\partial t} J(t,\theta)\mathrm{d}\theta=\int_{B_q(t)}\Delta f, \quad \quad \forall t> 0. \label{eq derivative of f negative}
    \end{align}
\end{lemma}

\pf
{\textbf{Step (1)}.  From Gauss' Lemma, we have 
     \begin{align}
    \frac{\partial}{\partial t}\tilde{f}(t, \theta)= \nabla f\cdot \frac{\partial}{\partial t}\exp_q(t\theta)= \nabla f\cdot \vec{n}= \frac{\partial f}{\partial \vec{n}}(x)  , \quad \quad \forall t\cdot \theta\in A, x= \exp_q(t\cdot \theta), \label{partial diff eq}
     \end{align}
     where $\vec{n}$ is the out-forward unit normal vector of $\partial B_q(t)$. 
     
      From the definition of $Z$, we have 
     \begin{align}
        \int_{\tilde\Sigma(t)}\frac{\partial \tilde{f}(t,\theta)}{\partial t} J(t,\theta)\mathrm{d}\theta &=\int_{\Sigma(t)}\frac{\partial \tilde{f}(t,\theta)}{\partial t} J(t,\theta)\mathrm{d}\theta \quad \quad \forall t\in (0,\infty)\backslash Z. \label{first eq about inte-for Z}
    \end{align} 
     
     Now from (\ref{chaning variable formula-slice}) and (\ref{partial diff eq}), we have
    \begin{align}
        \int_{\Sigma(t)}\frac{\partial \tilde{f}(t,\theta)}{\partial t} J(t,\theta)\mathrm{d}\theta= \int_{\exp_q(t\cdot\Sigma(t))}\frac{\partial f}{\partial \vec{n}}\mathrm{d}\mu_{\partial B_q(t)} . \label{first eq about inte-2}
    \end{align}
    
     By (\ref{first eq about inte-for Z}) and (\ref{first eq about inte-2}) and Lemma \ref{lem divergence for geodesic ball}, we get
      \begin{align}
         \int_{\tilde\Sigma(t)}\frac{\partial \tilde{f}(t,\theta)}{\partial t} J(t,\theta)\mathrm{d}\theta &= \int_{B_q(t)}\Delta f, \quad \forall t\in (0,\infty)\backslash Z. \label{first eq about inte}
    \end{align}
    
 From \cite[Theorem III.2.1]{CI}, we know $\mathcal{R}$ is continuous. Then using the definition of $\mathcal{R}$ and its continuous property, we get
 \begin{align}
 \lim_{s\rightarrow t-}\chi_{\tilde\Sigma(s)}(x)= \chi_{\tilde\Sigma(t)}(x), \quad \quad \forall x\in \mathbb{S}^{n- 1}. \label{point wise convergence of char set func}
\end{align}  
    
    Also note $\tilde{f}$ is a smooth functions and $J$ is bounded measurable function. By Lebesgue's Dominated Convergence Theorem and (\ref{point wise convergence of char set func}), we get
    \begin{align}
        \lim_{s\to t-}\int_{\tilde\Sigma(s)}\frac{\partial \tilde{f}(s,\theta)}{\partial s} J(s,\theta)\mathrm{d}\theta=\int_{\tilde\Sigma(t)}\frac{\partial \tilde{f}(t,\theta)}{\partial t} J(t,\theta)\mathrm{d}\theta. \label{2nd eq about inte}
    \end{align}
    
    Similarly we also have
    \begin{align}
        \lim_{s\to t-}\int_{B_q(s)}\Delta f=\int_{B_q(t)}\Delta f ,\label{3rd eq about inte}
    \end{align}
    where we use the definition $B_q(t)= \{x: d(q, x)< t\}$ to get $\displaystyle \lim_{s\rightarrow t-}\chi_{B_q(s)}(x)= \chi_{B_q(t)}(x)$.
    
    Therefore by (\ref{first eq about inte}), (\ref{2nd eq about inte}), (\ref{3rd eq about inte}) and $\mathcal{H}^1(Z)= 0$ (by Lemma \ref{lem divergence for geodesic ball}), we get (\ref{eq derivative of f negative}).
}
\qed

\section{Sharp mean value inequality}

The following is a version of Bishop-Gromov's volume comparison theorem, which we will use in later argument.
\begin{lemma}\label{lem derivative of J}
For $(M^n, g)$ with $Rc\geq 0$, we have
\begin{align}
t_1^{1-n} J(t_1, \theta)\geq t_2^{1-n} J(t_2, \theta), \quad \quad \forall 0< t_1\leq t_2\leq \mathcal{R}(\theta). \nonumber 
\end{align}
\end{lemma}

\pf
{From \cite[Corollary 2.2]{Li book}, we have
\begin{align}
        \frac{\partial (t^{1-n}J(t,\theta))}{\partial t}\leq 0, \quad \quad \forall 0< t< \mathcal{R}(\theta),\,\theta\in \mathbb{S}^{n- 1}. \nonumber 
    \end{align}

Then we get 
    \begin{align}
    t_1^{1-n} J(t_1, \theta)\geq t_2^{1-n} J(t_2, \theta), \quad \quad \forall 0< t_1\leq t_2< \mathcal{R}(\theta). \label{mono of volume ratio}
    \end{align}
    Note $J(t,\theta)$ is well defined and continuous for $t> 0$, from (\ref{mono of volume ratio}) we have 
     \begin{align}
    t_1^{1-n} J(t_1, \theta)\geq \lim_{t_2\rightarrow \mathcal{R}(\theta)^-}t_2^{1-n} J(t_2, \theta)= \mathcal{R}(\theta)^{1-n} J(\mathcal{R}(\theta), \theta), \quad \quad \forall 0< t_1< \mathcal{R}(\theta). \label{mono of volume ratio-2}
    \end{align}

From (\ref{mono of volume ratio}) and (\ref{mono of volume ratio-2}), we get
\begin{align}
&t_1^{1-n} J(t_1, \theta)\geq t_2^{1-n} J(t_2, \theta), \quad \quad \forall 0< t_1\leq t_2\leq \mathcal{R}(\theta).
\end{align}
Then the conclusion follows.
}
\qed

Instead of considering $t^{1-n}\int_{\partial B_q(t)}f(x)d\mu(x)$, we study the general term $F(t)$ defined in (\ref{def of F}) and show it is non-increasing. Note when $t< \mathrm{inj}(q)$, $F(t)$ is exactly $t^{1-n}\int_{\partial B_q(t)}f(x)d\mu(x)$, which is used in the proof of Schoen-Yau's original argument.
\begin{lemma}\label{lem decreasing of average integral}
    Suppose that $f\in C^\infty(M^n)$ satisfies $\Delta f\leq 0$ and $f\geq 0$, define
    \begin{align}
        F(t)=t^{1-n}\int_{\tilde\Sigma(t)}\tilde{f}(t,\theta) J(t,\theta)\mathrm{d}\theta, \label{def of F}
    \end{align}
then $F$ is non-increasing.        
\end{lemma}

\pf
{\textbf{Step (1)}.  For $t>s>0$, since $\tilde\Sigma(t)\subset\tilde\Sigma(t-s)$ for $s> 0$, note $\tilde{f}\geq 0$ and Lemma \ref{lem derivative of J}, we have
    \begin{align}
        &F(t)-F(t-s)\nonumber \\
        =&\int_{\tilde\Sigma(t)}\tilde{f}(t,\theta) t^{1-n}J(t,\theta)\mathrm{d}\theta-\int_{\tilde\Sigma(t-s)}\tilde{f}(t- s,\theta) (t-s)^{1-n}J(t-s,\theta)\mathrm{d}\theta\nonumber \\
        =&\int_{\tilde\Sigma(t)}\Big(\tilde{f}(t,\theta) t^{1-n}J(t,\theta)-\tilde{f}(t-s,\theta) (t-s)^{1-n}J(t-s,\theta)\Big)\mathrm{d}\theta\nonumber \\
        &+\int_{\mathbb{S}^{n-1}}\tilde{f}(t-s,\theta) (t-s)^{1-n}J(t-s,\theta)(\chi_{\tilde\Sigma(t)}-\chi_{\tilde\Sigma(t-s)})\mathrm{d}\theta \nonumber \\
   & \leq \int_{\tilde\Sigma(t)}\Big(\tilde{f}(t,\theta) t^{1-n}J(t,\theta)-\tilde{f}(t-s,\theta) (t-s)^{1-n}J(t-s,\theta)\Big)\mathrm{d}\theta \nonumber \\
   &\leq \int_{\tilde\Sigma(t)} [\tilde{f}(t, \theta)  - \tilde{f}(t-s, \theta)]\frac{J(t, \theta)}{t^{n- 1}}  d\theta+  \int_{\tilde\Sigma(t)} \tilde{f}(t-s, \theta)\cdot \left(  \frac{J(t, \theta)}{t^{n- 1}} - \frac{J(t-s, \theta)}{(t- s)^{n- 1}} \right) d\theta \nonumber \\
   &\leq \int_{\tilde\Sigma(t)} [\tilde{f}(t, \theta)  - \tilde{f}(t-s, \theta)] \frac{J(t, \theta)}{t^{n- 1}}  d\theta.\label{ineq about F}
    \end{align}
And we also have
\begin{align}
        &F(t+s)- F(t)\leq \int_{\tilde\Sigma(t)}\Big(\tilde{f}(t+s,\theta) (t+s)^{1-n}J(t+s,\theta)- \tilde{f}(t,\theta) t^{1-n}J(t,\theta)\Big)\mathrm{d}\theta \nonumber \\
   &\leq \int_{\tilde\Sigma(t)} [\tilde{f}(t+s,\theta) - \tilde{f}(t, \theta)  ]\frac{J(t, \theta)}{t^{n- 1}}  d\theta  .\label{ineq about F-other side}
    \end{align}
    
 \textbf{Step (2)}.   Assume $0< a_0< b_0$, we only need to show $F(b_0)- F(a_0)\leq 0$. By contradiction, assume there is $\epsilon> 0$ such that
 \begin{align}
 F(b_0)- F(a_0)> \epsilon\cdot [b_0-a_0]. \label{ineq about F a0 and b0}
\end{align}  
 
From (\ref{ineq about F a0 and b0}), for any $i\in \mathbb{Z}^+$, there is $[a_i, b_i]\subseteq [a_{i- 1}, b_{i- 1}]$ such that  
\begin{align}
&F(b_i)- F(a_i)\geq \epsilon\cdot [b_i- a_i], \label{restriction on F ai bi} \\
&b_i- a_i= \frac{b_{i- 1}- a_{i- 1}}{2}, \quad \text{and} \quad (a_i- a_{i- 1})\cdot (b_i- b_{i- 1})= 0, \nonumber 
\end{align}

From compactness of $[a_0, b_0]$, and $\displaystyle \lim_{i\rightarrow\infty}|b_i- a_i|= 0$, there is $\displaystyle a_\infty\in \bigcap_{i= 1}^\infty[a_i, b_i]$ such that 
\begin{align}
\lim_{i\rightarrow\infty}a_i= \lim_{i\rightarrow\infty}b_i= a_\infty. \label{limit of ai and bi}
\end{align}

By (\ref{restriction on F ai bi}) and (\ref{limit of ai and bi}), we get that there is a subsequence of $\{b_i\}$ or $\{a_i\}$ (for simplicity, also denoted as $\{b_i\}$ or $\{a_i\}$) such that 
\begin{align}
F(b_i)- F(a_\infty)\geq \epsilon\cdot [b_i- a_\infty] ,\label{Fbi ainfinity} \\
F(a_\infty)- F(a_i)\geq \epsilon\cdot [a_\infty- a_i] .\label{Fai ainfinity}
\end{align}

Without loss of generality, we can assume (\ref{Fbi ainfinity}) holds in the rest argument (if (\ref{Fai ainfinity}) holds, similar argument applies, we use (\ref{ineq about F}) instead).

Then from (\ref{Fbi ainfinity}) and (\ref{ineq about F-other side}), we get 
\begin{align}
&\frac{1}{[b_i- a_\infty]}\int_{\tilde\Sigma(a_\infty)} [\tilde{f}(b_i, \theta)  - \tilde{f}(a_\infty, \theta)] \frac{J(a_\infty, \theta)}{a_\infty^{n- 1}}  d\theta\geq \epsilon \label{ineq of integral-b}  . 
\end{align}

Taking the limit with respect to $i$ in (\ref{ineq of integral-b}), using (\ref{limit of ai and bi}) and Lebesgue's Dominated Convergence Theorem, we obtain
\begin{align}
\int_{\tilde\Sigma(a_\infty)} \frac{J(a_\infty, \theta)}{a_\infty^{n- 1}} \frac{\partial}{\partial t}\tilde{f}(a_\infty, \theta) d\theta\geq \epsilon> 0.\label{deriv inte is positive}
\end{align}

However from Lemma \ref{lem generalized divergence theorem for geodesic ball} and $\Delta f\leq 0$, we have
        \begin{align}
       \int_{\tilde\Sigma(a_\infty)} \frac{J(a_\infty, \theta)}{a_\infty^{n- 1}} \frac{\partial}{\partial t}\tilde{f}(a_\infty, \theta) d\theta= a_\infty^{1-n} \int_{B_q(a_\infty)}\Delta f \leq 0, \label{limit for integal-1}
        \end{align}
 it contradicts (\ref{deriv inte is positive}).
}
\qed

The following result is a generalization of \cite[Proposition $2.1$, Chapter $1$]{SY}.
\begin{theorem}[Sharp mean value inequality]\label{thm sharp mvi for supharm}
    Suppose that $f\in C^\infty(M^n)$ satisfies $\Delta f\leq 0$ and $f\geq 0$, then
        \begin{align}
        f(q)\geq \frac{1}{n\omega_n r^{n-1}}\int_{\partial B_q(r)} f, \quad \quad \quad \forall q\in (M^n, g), r>0. \label{sharp mvi for supharm}
    \end{align}
\end{theorem}

\begin{remark}\label{rem improvement on SY}
{In \cite{SY}, the above result is proved for any $r< \mathrm{inj}(q)$, where $\mathrm{inj}(q)$ is the injectivity radius of $q\in (M^n, g)$. Our result removes this restriction on $r$, furthermore Theorem \ref{thm rigidity of sharp mvi for supharm} provides the corresponding rigidity result when the equation is obtained.
}
\end{remark}

\begin{proof} 
Define $F$ as in Lemma \ref{lem decreasing of average integral}, from Lemma \ref{lem decreasing of average integral}, we get 
    \begin{align}
        F(t)\leq \lim_{r\to 0^+}F(r)=n\omega_n f(q), \quad \quad \forall t> 0. \label{upper bound of Fr}
    \end{align}
  
For any $t>0$, notice that $\partial B_q(t)= \exp_q(t\cdot\tilde\Sigma(t))$. From $f\geq 0$, (\ref{chaning variable formula-slice}) and (\ref{upper bound of Fr}), we have
    \begin{align*}
       &\quad t^{1-n}\int_{\partial B_q(t)}f \mathrm{d} \mu_{\partial B_q(t)}\leq F(t) \leq n\omega_n f(q).
    \end{align*}
\end{proof} 
 
\begin{theorem}[Rigidity of mean value inequality]\label{thm rigidity of sharp mvi for supharm}
    Suppose that $f\in C^\infty(M^n)- \{0\}$ with $f\geq 0$ and $\Delta f\leq 0$. If for some $q\in (M^n, g)$ and $r>0$,
    \begin{align}
        f(q)= \frac{1}{n\omega_n r^{n-1}}\int_{\partial B_q(r)} f, \nonumber 
    \end{align}
    then $B_q(r)$ is isometric to $B^n(r)\subseteq \mathbb{R}^n$. 
\end{theorem}   

\pf
{\textbf{Step (1)}. Suppose there is $f(x_0)= 0$ for some $x_0\in B_q(r)$, from $\Delta f\le 0, f\geq 0$ and the Strong Maximum Principle, we get that $f\equiv 0$, which is a contradiction.

Hence we have 
\begin{align}
f(x)> 0, \quad \quad \forall x\in B_q(r). \label{strict positive for f}
\end{align}

    Then    
    \[
    \lim_{t \to 0^+} F(t) = n \omega_n f(x) = r^{1-n} \int_{\partial B_x(r)} f \le F(r) \leq \lim_{t \to 0^+} F(t)
    \]
    hence
    \(
    F(t) \equiv n \omega_n f(x),  \forall \ 0 < t \leq r
    \). Assume \( r > t > s > 0 \) in the following argument, since
    \begin{align*}
    0&= \limsup_{s\to 0}\frac{F(t) - F(t-s)}{s} \\
      &\le  \limsup_{s\to 0}\frac{1}{s}\int_{{\tilde{\Sigma}}(t)} \left( f(\exp_x(t\theta)) \, t^{1-n} J(t, \theta)  - f(\exp_x((t-s)\theta)) \, t^{1-n} J(t, \theta) \right) d\theta \\
      &\quad + \limsup_{s\to 0}\frac{1}{s}\int_{{\tilde{\Sigma}}(t)} \left( f(\exp_x((t-s)\theta)) \, t^{1-n} J(t, \theta) - f(\exp_x((t-s)\theta)) (t-s)^{1-n} J(t-s, \theta) \right) d\theta \\
      &\quad + \limsup_{s\to 0}\frac{1}{s}\int_{S^{n-1}} (t-s)^{1-n} f(\exp_x((t-s)\theta)) J(t-s, \theta) (1_{{\tilde{\Sigma}}(t)}-1_{{\tilde{\Sigma}}(t-s)} )d\theta\\
      &:=(I)+(II)+(III).
    \end{align*} 

    Note the limit $(I)$ exists, and $(II)$ and $(III)$ are nonpositive, hence $(I)\ge 0$. Since by (\ref{eq derivative of f negative})
    \begin{align*}
    0 &\geq \int_{{\tilde{\Sigma}}(t)} t^{1-n} J(t, \theta) \frac{\partial f(\exp_x(t\theta))}{\partial t}  d\theta=(I)\ge 0.
    \end{align*}
    then we get $(I)= 0$, hence $(II)=(III)=0$. Since $f>0$, there exists $C_1>0$ such that  $0<C_1 <f$ uniformly on compact set $\overline{B_x(r)}$, we get 
    \begin{align*}
        0=& \limsup_{s\to 0}\frac{1}{s}\int_{{\tilde{\Sigma}}(t)} \left( f(\exp_x((t-s)\theta)) \, t^{1-n} J(t, \theta) - f(\exp_x((t-s)\theta)) (t-s)^{1-n} J(t-s, \theta) \right) d\theta \\
        \le &\limsup_{s\to 0}C_1\int_{{\tilde{\Sigma}}(t)} \frac{1}{s}(t^{1-n} J(t, \theta)-(t-s)^{1-n} J(t-s, \theta)) d\theta \\
        \le &0.
    \end{align*}
    then there exists subsequence $s_i\to 0$ such that 
    \begin{equation*}
        \lim_{i\to \infty}\frac{1}{s_i}(t^{1-n} J(t, \theta)-(t-s_i)^{1-n} J(t-s_i, \theta))=0, \quad \mathcal{H}^{n-1}-a.e. \ \mathrm{in} \ \tilde{\Sigma}(t).
    \end{equation*}
    which inplies
    \begin{equation*}
         \frac{\partial (t^{1-n}J(t,\theta))}{\partial t}=0,\quad \forall (t,\theta)\in A
    \end{equation*}
     by continuity. Hence $J(t,\theta)=t^{n-1}$ by $\lim\limits_{t\to 0} t^{1-n}J(t,\theta)=1$ on $\overline{A}$ by continuity again. 
     
\textbf{Step (2)}.      Then the limit
     \begin{equation*}
        \lim_{s\to 0}\frac{1}{s}\int_{{\tilde{\Sigma}}(t)} \left( f(\exp_x((t-s)\theta)) \, t^{1-n} J(t, \theta) - f(\exp_x((t-s)\theta)) (t-s)^{1-n} J(t-s, \theta) \right) d\theta=0,
     \end{equation*}
     Since $0= \lim_{s\to 0}\frac{F(t) - F(t-s)}{s}$, we get the existence of following limit,
     \begin{align*}
       0&= \lim_{s\to 0}\frac{1}{s}\int_{S^{n-1}} (t-s)^{1-n} f(\exp_x((t-s)\theta)) J(t-s, \theta) (1_{{\tilde{\Sigma}}(t)}-1_{{\tilde{\Sigma}}(t-s)} )d\theta\\
        &=\lim_{s\to 0}\frac{1}{s}\int_{S^{n-1}}  f(\exp_x((t-s)\theta))  (1_{{\tilde{\Sigma}}(t)}-1_{{\tilde{\Sigma}}(t-s)} )d\theta\\
        &\le \lim_{s\to 0}\frac{C_1}{s}\int_{S^{n-1}}(1_{{\tilde{\Sigma}}(t)}-1_{{\tilde{\Sigma}}(t-s)})d\theta\\
        &\le 0,
     \end{align*}
     which implies for any $t\in (0,r)$, any $\epsilon>0$, there exists $\delta_t>0$, such that for any $s\in (t-\delta_t,t)$, there is 
     \begin{equation*}
        \mathrm{Area}(\tilde{\Sigma}(t-s))\ge \mathrm{Area}(\tilde{\Sigma}(t))-\epsilon s.
     \end{equation*}
     Similarly, for any $s\in (t,t+\delta_t)$, there is 
     \begin{equation*}
        \mathrm{Area}(\tilde{\Sigma}(t+s))\ge \mathrm{Area}(\tilde{\Sigma}(t))-\epsilon s.
     \end{equation*}

     Then suppose $\tilde{r}\in (0,r)$, for any $\epsilon>0$, by the compactness of $[0,\tilde{r}]$,
      we can choose $\delta >0$, and a division $0=r_0<r_1<\cdots <r_N=\tilde{r}$, such that $r_{i+1}-r_i=\delta$, 
    $\Sigma(r_1)=S^{n-1}$,  then 
    \begin{equation*}
        \mathrm{Area}(\tilde{\Sigma}(r_i+s))\ge \mathrm{Area}(\tilde{\Sigma}(r_i))-\epsilon \delta, \quad \forall s\in [r_{i},r_{i+1}],1\le i\le N-1.
    \end{equation*} 
    then we have 
     \begin{align*}
        \mathrm{Vol}(B_x(\tilde{r}))&=\int_{0}^{\tilde{r}}\int_{\tilde{\Sigma}(t)}J(t,\theta)d\theta dt=\int_{0}^{\tilde{r}}\int_{\tilde{\Sigma}(t)}t^{n-1}d\theta dt\\
        &=\sum_{i=1}^{N-1}\int_{r_i}^{r_{i+1}}\int_{\tilde{\Sigma}(t)}t^{n-1} d\theta dt+\int_0^{r_1}\int_{S^{n-1}}t^{n-1}d\theta dt\\
        &\ge\sum_{i=1}^{N-1}\int_{r_i}^{r_{i+1}}t^{n-1} (n\omega_n -\delta\epsilon i)dt+ \int_0^{r_1}\int_{S^{n-1}}t^{n-1}d\theta dt\\
        &\ge \omega_n\tilde{r}^n- C(n,r)\epsilon \delta^2 \frac{(N-1)N}{2}\\
        &= \omega_n\tilde{r}^n- C(n,r)\epsilon \left(\frac{ \tilde{r}}{N}\right)^2 \frac{(N-1)N}{2}\\
        &\ge \omega_n\tilde{r}^n- C(n,r)\epsilon.
     \end{align*}
        Let $\epsilon\to 0$, we get $\mathrm{Vol}(B_x(\tilde{r}))\ge \omega_n \tilde{r}^n$, by Gromov-Bishop volume comparison, 
     $\mathrm{Vol}(B_x(\tilde{r}))= \omega_n \tilde{r}^n$, then let $\tilde{r}\to r$, we get 
     \begin{equation*}
        \mathrm{Vol}(B_x(r))= \omega_n r^n,
     \end{equation*}
     by rigidity part of Gromov-Bishop volume comparison, we get $B_x(r)$ is isometric to $B^n(r)$.
}
\qed

\section*{Part II. Integral of curvature}

In the rest of the paper, we use $b$ instead of $b_q$ for simplicity when the context is clear.

The following Lemma was essentially proved in \cite{CM-AJM} firstly, which used Gromov-Hausdorff convergence. Our statement followed from the intrinsic argument in \cite{LTW}.

\begin{lemma}\label{lem point-wise bound of b}
{If $M^n$ has $Rc\geq 0$ with $n\geq 3$ and maximal volume growth, for any $\delta \in (0, \frac{1}{2}]$, we have 
\begin{align}
\big(\mathrm{V}_M\big)^{\frac{1}{n- 2}}\Big(1+ \tau\Big)^{\frac{1}{2- n}}\rho(x) \leq b(x)\leq \big(\mathrm{V}_M\big)^{\frac{1}{n- 2}}\Big(1- \tau\Big)^{\frac{1}{2- n}}\rho(x) \nonumber 
\end{align}
where $\rho(x)= d(q, x)$ and
\begin{align}
&\theta_q(r)= \frac{\mathrm{Vol}(\partial B_r(q))}{r^{n- 1}}, \quad \quad \theta_\infty= \lim\limits_{r\rightarrow \infty}\theta_q(r), \nonumber \\
&\tau= C(n)\big[\delta+ (\theta_q(\delta \rho(x))- \theta_\infty)^{\frac{1}{n- 1}}\big]. \nonumber 
\end{align}
Especially, $\lim\limits_{\rho(x)\rightarrow \infty} \frac{b(x)}{\rho(x)}= \big(\mathrm{V}_M\big)^{\frac{1}{n- 2}}$.
}
\end{lemma}\qed

\begin{lemma}\label{lem upper bound of gradient of b}
{If $M^n$ has $Rc\geq 0$ with $n\geq 3$, and it is non-parabolic, then $|\nabla b|\leq 1$.
}
\end{lemma}

\pf
{The conclusion follows from \cite[Lemma $2.2$]{Xu-IOC1} and \cite[Theorem $3.1$]{Colding}.
} 
\qed

The following lemma was implied by the argument in \cite{CC-Ann}, and was used repeatedly in \cite{CM-AJM}, \cite{Colding} and \cite{CM}. We give a direct proof of this result here for reader's convenience.
\begin{cor}\label{cor limit of the rescaled volume by b}
{For a complete non-compact Riemannian manifold $M^n$, with $Rc\geq 0$, $n\ge 3$ and maximal volume growth, we have $\lim\limits_{r\rightarrow \infty}\frac{\int_{b\leq r}|\nabla b|^3}{r^n}= \big(\mathrm{V}_M\big)^{\frac{1}{n- 2}}\omega_n$.
}
\end{cor}

\pf
{\textbf{Step (1)}. Let $\tilde{r}= \big(\mathrm{V}_M\big)^{\frac{1}{2- n}}r, \hat{b}= \big(\mathrm{V}_M\big)^{\frac{1}{2- n}}b$.
 
Firstly we recall that $\int_{b\leq r}|\nabla b|^2= \omega_nr^n$ by Co-area formula and (\ref{slice weighted volume}), which implies
\begin{align}
\int_{\hat{b}\leq r}|\nabla \hat{b}|^2= \mathrm{V}_M\omega_n r^n \nonumber 
\end{align}

From Green's formula and $\Delta b= (n- 1)\frac{|\nabla b|^2}{b}$, we get
\begin{align}
\int_{\hat{b}\leq r}\nabla \hat{b}\cdot \nabla \rho&= -\int_{\hat{b}\leq r}\Delta \hat{b}\cdot \rho+ \int_{\hat{b}= r}|\nabla \hat{b}|\cdot \rho \nonumber \\
& = -(n- 1)\int_{\hat{b}\leq r}\rho\frac{|\nabla \hat{b}|^2}{\hat{b}}+ \int_{\hat{b}= r}|\nabla \hat{b}|r+ \int_{\hat{b}= r}|\nabla \hat{b}|\cdot (\rho- r). \nonumber 
\end{align}
Using $\int_{b= r}|\nabla b|= n\omega_n r^{n- 1}$, we have
\begin{align}
\int_{\hat{b}= r}|\nabla \hat{b}|r= \mathrm{V}_Mn\omega_n r^{n}. \nonumber 
\end{align}

Note by Lemma \ref{lem point-wise bound of b} and Lemma \ref{lem upper bound of gradient of b}, the following holds
\begin{align}
\lim_{r\rightarrow\infty} \frac{\int_{\hat{b}\leq r}\rho\frac{|\nabla \hat{b}|^2}{\hat{b}}- |\nabla \hat{b}|^2}{V(\hat{b}\leq r)}= 0 \quad \quad \text{and} \quad \quad \lim_{r\rightarrow\infty} \frac{\int_{\hat{b}= r}|\nabla \hat{b}|\cdot (\rho- r)}{V(\hat{b}\leq r)}= 0. \nonumber 
\end{align}

By the above, using Lemma \ref{lem point-wise bound of b} again,
\begin{align}
\lim_{r\rightarrow\infty} \frac{\int_{\hat{b}\leq r}\nabla \hat{b}\cdot \nabla \rho}{V(\hat{b}\leq r)}&= -(n- 1)\lim_{r\rightarrow\infty}\frac{\int_{\hat{b}\leq r}|\nabla \hat{b}|^2}{V(\hat{b}\leq r)}+ \lim_{r\rightarrow\infty}\frac{\mathrm{V}_Mn\omega_nr^n}{V(\hat{b}\leq r)} \nonumber \\
&= \lim_{r\rightarrow\infty}\frac{\mathrm{V}_M\omega_nr^n}{V(\hat{b}\leq r)}= 1. \nonumber
\end{align}

Finally, 
\begin{align}
\lim\limits_{r\rightarrow \infty}\frac{\int_{\hat{b}\leq r} |\nabla \hat{b}- \nabla \rho|^2}{V(\hat{b}\leq r)}&= \lim\limits_{r\rightarrow \infty}\frac{\int_{\hat{b}\leq r} |\nabla \hat{b}|^2}{V(\hat{b}\leq r)}+ \lim\limits_{r\rightarrow \infty}\frac{\int_{\hat{b}\leq r} |\nabla \rho|^2}{V(\hat{b}\leq r)}- 2\lim_{r\rightarrow\infty} \frac{\int_{\hat{b}\leq r}\nabla \hat{b}\cdot \nabla \rho}{V(\hat{b}\leq r)} \nonumber \\
&= \lim\limits_{r\rightarrow \infty}\frac{\mathrm{V}_M\omega_nr^n}{V(\hat{b}\leq r)}+ 1- 2= 0 . \label{volume diff limit is 0} 
\end{align}

\textbf{Step (2)}. Note we have 
\begin{align}
\frac{\int_{b\leq r}\Big||\nabla b|^3- \big(\mathrm{V}_M\big)^{\frac{3}{n- 2}}\Big|}{r^n}&= \big(\mathrm{V}_M\big)^{\frac{3- n}{n- 2}}\frac{\int_{\hat{b}\leq \tilde{r}}\big||\nabla \hat{b}|^3- 1\big|}{\tilde{r}^n}. \nonumber 
\end{align}

From (\ref{volume diff limit is 0}), Lemma \ref{lem upper bound of gradient of b} and Lemma \ref{lem point-wise bound of b}, using the Bishop-Gromov Volume Comparison Theorem, we get
\begin{align}
\varlimsup_{r\rightarrow \infty}\frac{\int_{\hat{b}\leq r}\big||\nabla \hat{b}|^3- 1\big|}{r^n}&\leq C(\mathrm{V}_M)\varlimsup_{r\rightarrow \infty}\frac{\int_{\hat{b}\leq r} \big||\nabla \hat{b}|- 1\big|}{r^n}\leq C\cdot \varlimsup_{r\rightarrow \infty}\frac{\int_{\hat{b}\leq r} \big|\nabla \hat{b}- \nabla \rho\big|}{r^n} \nonumber \\
&\leq C\cdot \varlimsup_{r\rightarrow \infty}\frac{\big(\int_{\hat{b}\leq r} \big|\nabla \hat{b}- \nabla \rho\big|^2\big)^{\frac{1}{2}}}{r^n}\cdot V(\hat{b}\leq r)^{\frac{1}{2}} \nonumber \\
&\leq C\cdot \Big(\varlimsup_{r\rightarrow \infty}\frac{\int_{\hat{b}\leq r} \big|\nabla \hat{b}- \nabla \rho\big|^2}{V(\hat{b}\leq r)}\Big)^{\frac{1}{2}}= 0, \nonumber 
\end{align}
which implies $\lim\limits_{r\rightarrow \infty}\frac{\int_{\hat{b}\leq r}\big||\nabla \hat{b}|^3- 1\big|}{r^n}= 0$. 

Hence from the above and Lemma \ref{lem point-wise bound of b}, we have 
\begin{align}
\lim\limits_{r\rightarrow \infty}\frac{\int_{b\leq r}|\nabla b|^3}{r^n}= \big(\mathrm{V}_M\big)^{\frac{3}{n- 2}}\lim\limits_{r\rightarrow \infty}\frac{V(b\leq r)}{r^n}=  \big(\mathrm{V}_M\big)^{\frac{3- n}{n- 2}}\lim\limits_{r\rightarrow \infty}\frac{V(\rho\leq r)}{r^n}= \big(\mathrm{V}_M\big)^{\frac{1}{n- 2}}\omega_n .\nonumber 
\end{align}
}
\qed

We will repeatedly use the following variant of the Bochner formula, which we state as a lemma for convenience.
\begin{lemma}\label{lem Bochner formula variant}
{For any smooth function $v$, we have 
\begin{equation}
\Delta|\nabla v| = \frac{|\nabla^2 v|^2- |\nabla^2v(\frac{\nabla v}{|\nabla v|})|^2}{|\nabla v|} + \langle \frac{\nabla v}{|\nabla v|}, \nabla(\Delta v) \rangle + \operatorname{Ric}(\frac{\nabla v}{|\nabla v|}, \frac{\nabla v}{|\nabla v|})\cdot |\nabla v|, \label{Lap of norm of grad of v}
\end{equation}
}
\end{lemma}

\pf
{From the standard Bochner formula for $|\nabla v|$,
\[
\frac{1}{2} \Delta(|\nabla v|^2) = |\nabla^2 v|^2 + \langle \nabla v, \nabla(\Delta v) \rangle + \operatorname{Ric}(\nabla v, \nabla v).
\]

Note we have
\[
\Delta|\nabla v| = \frac{1}{2|\nabla v|} \Delta(|\nabla v|^2) - \frac{|\nabla|\nabla v||^2}{|\nabla v|}.
\]

And from  \
\begin{align*}
    \nabla |\nabla v|^2 = 2\langle \nabla\nabla v,\nabla v\rangle=2\nabla^2v(\nabla v),
\end{align*}
we get $\nabla |\nabla v| = \nabla^2v(\frac{\nabla v}{|\nabla v|})$.

Thus
\begin{align*}
\Delta|\nabla v| &= \frac{1}{|\nabla v|} \left[ |\nabla^2 v|^2 + \langle \nabla v, \nabla(\Delta v) \rangle + \operatorname{Ric}(\nabla v, \nabla v) \right] - \frac{|\nabla^2 v(\frac{\nabla v}{|\nabla v|})|^2}{|\nabla v|} \\
&= \frac{|\nabla^2 v|^2 - |\nabla^2 v(\frac{\nabla v}{|\nabla v|})|^2}{|\nabla v|} + \frac{\langle \nabla v, \nabla(\Delta v) \rangle}{|\nabla v|} + \frac{\operatorname{Ric}(\nabla v, \nabla v)}{|\nabla v|}.
\end{align*}
}
\qed

\section{Weighted volume and its asymptotic behavior}

Define 
\begin{align}
&\mathcal{S}\vcentcolon= \{t\in \mathbb{R}^+: |\nabla b(x)|= 0\ \text{for some}\ x\in b^{-1}(t)\}, \nonumber 
\end{align}
From Sard's theorem we know  $\mathcal{L}^1(\mathcal{S})= 0$.

We define the weighted volume of $\{x\in M^n: b(x)= t\}$ as follows:
\begin{align}
\mathcal{A}(t)\vcentcolon= t^{1- n}\int_{b= t}|\nabla b|^2, \quad \quad \forall t\in \mathbb{R}^+. \nonumber 
\end{align}

We always assume $\vec{n}= \frac{\nabla b}{|\nabla b|}$ in this section.

\begin{definition} 
Let $E \subset M$, and let $v(x)$ be a unit vector in $T_xM$.  
Define $H^+_r(x) = \{ w \in T_x M : |w|<r, v(x) \cdot w \geq 0 \}$, if  
\begin{align*}
    \lim_{r \to 0} \frac{\mathcal{H}^n(E\cap \exp_x(H^+_r(x)))}{r^n} = 0,
\end{align*}
then $v(x)$ is called the \textit{measure theoretic unit outer normal} to $E$ at $x$.
\end{definition}

   \begin{lemma}\label{lem derivative of Vp}
{We have $\mathcal{A}'\in C(\mathbb{R}^+)$ and 
\begin{align}
    \mathcal{A}'(r)&=r^{1-n}\int_{b=r}\langle\nabla |\nabla b|,\vec{n}\rangle, \quad  \forall r> 0,\nonumber
\end{align}
    \begin{align}
\frac{\mathcal{A}'(t)}{t^{n- 3}}- \frac{\mathcal{A}'(s)}{s^{n- 3}}&\geq \int_{s< b< t} \Big\{\frac{\operatorname{Ric}(\vec{n}, \vec{n})+ |\Pi_0|^2+ \frac{|\nabla^2 b(\vec{n}, \vec{n})|^2}{(n- 1)|\nabla b|^2}}{b^{2n- 4}}\Big\}|\nabla b|\geq 0, \quad 0< s< t< \infty. \nonumber 
\end{align}
  }
\end{lemma}

\begin{remark}\label{rem why A'overtn-3}
{The strategy of deriving geometric information by examining the derivative of a primary monotone quantity has deep roots in geometric analysis. 

Classic examples include the monotonicity formula for minimal surfaces, where the derivative of the area ratio controls the second fundamental form, and Almgren's frequency function for harmonic functions, whose derivative measures the deviation from homogeneity. 

In a similar spirit, the analysis of the quantity $\frac{\mathcal{A}'(t)}{t^{n-3}}$ and its monotonicity serves as a higher-order probe into the geometry of the level sets of $b$. 

Its non-decreasing property, which follows directly from the condition $\operatorname{Ric} \ge 0$, packages the combined effects of Ricci curvature, the traceless second fundamental form, and the Hessian of $b$ into a single, powerful differential inequality. This formalism is the engine that ultimately drives the asymptotic integral estimates for the scalar curvature.
}
\end{remark}

\begin{proof}
\textbf{Step (1)}. For any $r> 0$, we define 
\begin{align}
Z(r)\vcentcolon= \{x\in M^n: b(x)= r, |\nabla b|(x)= 0\}. \nonumber 
\end{align}
   By \cite[Theorem $2.2$]{Cheng}, we know 
    \begin{align}
        \mathcal{H}^{n-1}(b^{-1}(r))<\infty, \quad \quad \mathcal{H}^{n- 1}(Z(r))= 0, \quad \quad \forall r> 0. \label{critical point has 0 measure}
    \end{align}

    For $t>s\geq0$, define $D=\{x\in M:s<b(x)<t\}$. Let $\partial_*D$ be defined as \eqref{measure theoretic boundary},

     Therefore we get
     \begin{align}
       & \mathcal{H}^{n-1}(\partial_* D)\leq \mathcal{H}^{n-1}\big(b^{-1}(s)\cup b^{-1}(t)\big)<\infty, \nonumber 
       \end{align}
       which implies $D$ is a set with finite perimeter by \cite[Theorem 5.23]{EG}.

Let $\vec{n}=\frac{\nabla b}{|\nabla b|}$ when $\nabla b\neq 0$ and $\vec{n}=0$ if $\nabla b= 0$. Let $\vec{v}$ be the measure theoretic unit outer normal of $\partial_* D$. Note that $\vec{v}=\vec{n}$ on $b^{-1}(t)\backslash Z(t)$ and $\vec{v}=-\vec{n}$ on $b^{-1}(s)\backslash Z(s)$ from the smoothness. Now by {\cite[Theorem 5.16]{EG}, we get
\begin{align}
\int_{\partial_*D}\langle b^{1-n}|\nabla b|\nabla b, \vec{v}\rangle= \int_{D}\mathrm{div}(b^{1-n}|\nabla b|\nabla b).\label{divergence application for A}
\end{align}

From the definition of $\partial_*D$ and the Implicit Function Theorem, we have 
    \begin{align}
     (b^{-1}(s)\backslash Z(s))\cup (b^{-1}(t)\backslash Z(t)) \subset\partial_* D\subset\partial D= b^{-1}(s)\cup b^{-1}(t), \nonumber
     \end{align}
this implies
       \begin{align}
        \mathcal{H}^{n-1}(b^{-1}(s)\cup b^{-1}(t)\backslash \partial_*D)=0.\label{measure boundary of D}
    \end{align}
   
    For $ t>s\geq 0$, by (\ref{measure boundary of D}) and (\ref{divergence application for A}), we get 
    \begin{align*}
        \mathcal{A}(t)-\mathcal{A}(s) =&\int_{b^{-1}(t)\backslash Z(t)}\langle b^{1-n}|\nabla b|\nabla b, \vec{n}\rangle-\int_{b^{-1}(s)\backslash Z(s)}\langle b^{1-n}|\nabla b|\nabla b, \vec{n}\rangle\\
        =&\int_{\partial_*D}\langle b^{1-n}|\nabla b|\nabla b, \vec{v}\rangle=\int_{D}\mathrm{div}(b^{1-n}|\nabla b|\nabla b)\\
        =&\int_{D}(1-n)b^{-n}|\nabla b|^3+b^{1-n}\langle\nabla |\nabla b|,\nabla b\rangle+b^{1-n}|\nabla b|\Delta b
    \end{align*}
    
    Using $\Delta b=(n-1)\frac{|\nabla b|^2}{b}$ and Co-area formula, we get
    \begin{align}
        \mathcal{A}(t)-\mathcal{A}(s)=&\int_{s< b < t}b^{1-n}\langle\nabla |\nabla b|,\nabla b\rangle
        =\int_s^t\mathrm{d}r\int_{b=r}b^{1-n}\langle\nabla |\nabla b|,\vec{n}\rangle. \label{difference between A volume}
    \end{align}
    
    \textbf{Step (2)}. Now we compute $\mathrm{div}\big(\frac{\nabla|\nabla b|}{b^{2n- 4}}\big)$ as follows:
\begin{align}
\mathrm{div}\big(\frac{\nabla|\nabla b|}{b^{2n- 4}}\big)= \frac{\Delta |\nabla b|}{b^{2n- 4}}- (2n- 4)\frac{\nabla b\cdot \nabla |\nabla b|}{b^{2n- 3}}. \nonumber 
\end{align}

Using Lemma \ref{lem Bochner formula variant} and (\ref{equation of Laplace of v}), we obtain
\begin{align}
\mathrm{div}\big(\frac{\nabla|\nabla b|}{b^{2n- 4}}\big)&= \frac{1}{b^{2n- 4}}\cdot \Big\{\frac{|\nabla^2 b|^2- |\nabla^2b(\vec{n})|^2}{|\nabla b|} + \langle \vec{n}, \nabla(\Delta b) \rangle + \operatorname{Ric}(\vec{n}, \vec{n})\cdot |\nabla b|\Big\}\nonumber \\
&\quad + (4- 2n)\frac{|\nabla b|}{b^{2n- 3}}\cdot \nabla^2 b(\vec{n}, \vec{n}) \nonumber \\
&= \frac{1}{b^{2n- 4}}\cdot \Big\{\frac{|\nabla^2 b|^2- |\nabla^2b(\vec{n})|^2}{|\nabla b|} + \operatorname{Ric}(\vec{n}, \vec{n})\cdot |\nabla b| \nonumber \\
&\quad + \langle \vec{n}, \nabla\big[(n-1)\frac{|\nabla b|^2}{b} \big] \rangle \Big\} + (4- 2n)\frac{|\nabla b|}{b^{2n- 3}}\cdot \nabla^2 b(\vec{n}, \vec{n}) \nonumber \\
&=\frac{\operatorname{Ric}(\vec{n}, \vec{n})\cdot |\nabla b|}{b^{2n- 4}}+ 2\frac{|\nabla b|\cdot \nabla^2 b(\vec{n}, \vec{n})}{b^{2n- 3}}+ \frac{|\nabla^2 b|^2- |\nabla^2b(\vec{n})|^2}{b^{2n- 4}|\nabla b|}- \frac{n- 1}{b^{2n- 2}}|\nabla b|^3.\label{integrand in 2nd integral}
\end{align}

From $\frac{|\nabla b|^2}{b}= \frac{1}{n- 1}\Delta b$ and (\ref{integrand in 2nd integral}), we get
\begin{align}
&\mathrm{div}\big(\frac{\nabla|\nabla b|}{b^{2n- 4}}\big)= \frac{\operatorname{Ric}(\vec{n}, \vec{n})\cdot |\nabla b|^2+ |\nabla^2 b|^2- |\nabla^2b(\vec{n})|^2+ 2\frac{\nabla^2 b(\vec{n}, \vec{n})}{b}|\nabla b|^2- \frac{n- 1}{b^{2}}|\nabla b|^4}{b^{2n- 4}|\nabla b|} \nonumber \\
&=\frac{\operatorname{Ric}(\vec{n}, \vec{n})\cdot |\nabla b|^2+  |\nabla^2 b|^2- |\nabla^2b(\vec{n})|^2+ 2\frac{\Delta b}{n- 1}\nabla^2 b(\vec{n}, \vec{n})- \frac{(\Delta b)^2}{n- 1}}{b^{2n- 4}|\nabla b|}\cdot \nonumber  \\
&\geq \frac{\operatorname{Ric}(\vec{n}, \vec{n})\cdot |\nabla b|^2+  \sum\limits_{i= 1}^{n- 1}|\nabla^2 b(i, i)|^2- \frac{\big(\sum\limits_{i= 1}^{n- 1}\nabla^2 b(i, i)\big)^2}{n- 1}+ \frac{|\nabla^2 b(\vec{n}, \vec{n})|^2}{n- 1}}{b^{2n- 4}|\nabla b|}\nonumber\\
&= \frac{\operatorname{Ric}(\vec{n}, \vec{n})\cdot |\nabla b|^2+  \sum\limits_{i= 1}^{n- 1}[\nabla^2 b(i, i)- \frac{\sum_{j= 1}^{n- 1}\nabla^2 b(j, j)}{n- 1}]^2+ \frac{|\nabla^2 b(\vec{n}, \vec{n})|^2}{n- 1}}{b^{2n- 4}|\nabla b|} \nonumber \\
&=\frac{\operatorname{Ric}(\vec{n}, \vec{n})\cdot |\nabla b|^2+  |\prod_0|^2\cdot |\nabla b|^2+ \frac{|\nabla^2 b(\vec{n}, \vec{n})|^2}{n- 1}}{b^{2n- 4}|\nabla b|}\geq 0, \label{non-negative 2nd integrand}
\end{align}
where $|\Pi_0|$ is the trace-free second fundamental form of $b^{-1}(t)$.

Define $\hat{A}(r)= \int_{b=r}b^{4- 2n}\langle\nabla |\nabla b|,\vec{n}\rangle$, then similar to the Step (1) and applying divergence theorem (we use the similar approximation as in \cite[Lemma $3.7$]{AFM} such that the divergence theorem can be applied), we obtain
\begin{align}
\hat{A}(t)- \hat{A}(s)&= \int_{s< b< t} \mathrm{div}\big(\frac{\nabla|\nabla b|}{b^{2n- 4}}\big). \label{first deri difference formula}
\end{align}

\textbf{Step (3)}. Note 
\begin{align}
|\hat{A}(r)|\leq \int_{b=r }b^{4-2n}|\nabla^2 b|< \infty, \quad \quad  \forall r> 0. \label{upper bound of hat A}
\end{align}
Hence from (\ref{first deri difference formula}), (\ref{non-negative 2nd integrand}) and (\ref{upper bound of hat A}), we get 
\begin{align}
\int_{s< b< t} |\mathrm{div}\big(\frac{\nabla|\nabla b|}{b^{2n- 4}}\big)|= \int_{s< b< t} \mathrm{div}\big(\frac{\nabla|\nabla b|}{b^{2n- 4}}\big)= \hat{A}(t)- \hat{A}(s)< \infty, \quad \quad \forall 0< s\leq t< \infty. \label{integral expression of hat A}
\end{align}

From Co-area formula and Fubini's Theorem, we get
\begin{align}
\hat{A}(t)- \hat{A}(s)= \int_{s< r< t}dr\int_{b= r}  \frac{1}{|\nabla b|}\cdot \mathrm{div}\big(\frac{\nabla|\nabla b|}{b^{2n- 4}}\big), \label{expression of hat A}
\end{align}
Now from the continuity of the integral with respect to variable $s, t$, using (\ref{expression of hat A}) we get that $\hat{A}: \mathbb{R}^+\rightarrow \mathbb{R}$ is continuous. 

Now from (\ref{difference between A volume}), we get
\begin{align}
\mathcal{A}(t)- \mathcal{A}(s)= \int_s^t r^{n-3}\hat{A}(r)dr, \label{expression of A by hat A}
\end{align}

From the continuity of $\hat{A}$ and (\ref{expression of A by hat A}), we obtain
    \begin{align}
        \mathcal{A}'(r)= \hat{A}(r) =\int_{b=r}b^{1-n}\langle\nabla |\nabla b|,\vec{n}\rangle, \quad \quad \forall r> 0. \nonumber 
    \end{align}
    }
\end{proof}

\begin{lemma}\label{lem mean curv of level set in form of bp}
{Assume $H$ is the mean curvature of the level set $\{x\in M^n: b(x)= t\}\subseteq M^n$, then  
\begin{align}
H= \frac{(n- 1)|\nabla b|}{b}- \frac{\nabla^2 b(\vec{n}, \vec{n})}{|\nabla b|}. \nonumber 
\end{align}
}
\end{lemma}

\pf
{From the definition of mean curvature directly, for any smooth function $v$, we have 
\begin{align}
H= \frac{\Delta v}{|\nabla v|}- \frac{\nabla^2v(\vec{n}, \vec{n})}{|\nabla v|}, \label{mean curv formula}
\end{align}
where $H$ is the mean curvature of the level set $\{x: v(x)= t\}$. 

The conclusion follows from (\ref{mean curv formula}) and (\ref{equation of Laplace of v}).
}
\qed

\begin{cor}\label{cor monotonicity of Vt}
{For $0< t_1\leq t_2$, we have
\begin{align}
&\mathcal{A}'(t_1)\leq \mathcal{A}'(t_2)\leq 0, \nonumber \\
&\lim_{t\rightarrow\infty\atop t\in \mathbb{R}^+}\mathcal{A}'(t)= 0. \nonumber 
\end{align}
}
\end{cor}

\pf
{\textbf{Step (1)}. We assume $\displaystyle c_1\vcentcolon= \varliminf_{t\rightarrow\infty}\mathcal{A}'(t)$, note $c_1\in [-\infty, \infty]$. We will prove $c_1\leq 0$.

By contradiction, if $c_1> 0$, then there is $\epsilon> 0$ and $t_0> 0$ such that
\begin{align}
\inf_{t\geq t_0}\mathcal{A}'(t)\geq \epsilon. \nonumber 
\end{align}

Now for any $t_1> t_0$, we have 
\begin{align}
\mathcal{A}(t_1)= \mathcal{A}(t_0)+ \int_{t_0}^{t_1}\mathcal{A}'(t)dt\geq \mathcal{A}(t_0)+ \epsilon\cdot [t_1- t_0]. \nonumber 
\end{align}

So by Lemma \ref{lem upper bound of gradient of b} and (\ref{slice weighted volume}),
\begin{align*}
    \mathcal{A}(t_1)\leq t_1^{1-n}\int_{b=t_1}|\nabla b|=\lim_{r\to 0}r^{1-n}\int_{b=r}|\nabla b|=n\omega_n.
\end{align*}

Therefore we obtain
\begin{align}
n\omega_n\geq \mathcal{A}(t_0)+ \epsilon\cdot [t_1- t_0]. \nonumber 
\end{align}

Therefore we have $\displaystyle \frac{n\omega_n- \mathcal{A}(t_0)}{t_1- t_0}\geq \epsilon$. Let $t_1\rightarrow\infty$, we get 
\begin{align}
0\geq \epsilon,\nonumber 
\end{align}
which is a contradiction. Therefore we get $c_1\leq 0$.

\textbf{Step (2)}. From $\displaystyle\varliminf_{t\rightarrow\infty}\mathcal{A}'(t)= c_1\leq 0$, we get 
\begin{align}
\varliminf_{t\rightarrow\infty}\frac{\mathcal{A}'(t)}{t^{n- 3}}\leq 0. \label{liminf of V' and tn-3}
\end{align}

From Lemma \ref{lem derivative of Vp}, we know that $\frac{\mathcal{A}'(t)}{t^{n- 3}}$ is a non-decreasing function in $t$, hence the limit $\displaystyle \lim_{t\rightarrow\infty}\frac{\mathcal{A}'(t)}{t^{n- 3}}$ exists.

Now by (\ref{liminf of V' and tn-3}), we get
\begin{align}
\lim_{t\rightarrow\infty}\frac{\mathcal{A}'(t)}{t^{n- 3}}= \varliminf_{t\rightarrow\infty}\frac{\mathcal{A}'(t)}{t^{n- 3}}\leq 0. \nonumber 
\end{align}

By Lemma \ref{lem derivative of Vp} again, we get
\begin{align}
\frac{\mathcal{A}'(s)}{s^{n- 3}}\leq \lim_{t\rightarrow\infty}\frac{\mathcal{A}'(t)}{t^{n- 3}}\leq 0, \quad \quad \forall s\in \mathbb{R}^+. \nonumber 
\end{align}

From Lemma \ref{lem derivative of Vp}, we know that $\frac{\mathcal{A}'(t)}{t^{n- 3}}$ is a non-decreasing function in $t$, and
\begin{align}
\mathcal{A}'(t_1)\leq \frac{t_1^{n- 3}}{t_2^{n- 3}}\mathcal{A}'(t_2)\leq \mathcal{A}'(t_2), \quad \quad \forall 0< t_1\leq t_2< \infty. \label{increasing of A'}
\end{align}

Therefore the limit $\displaystyle c_2\vcentcolon= \lim_{t\rightarrow\infty}\mathcal{A}'(t)\in [-\infty, 0]$ exists.

\textbf{Step (3)}. If $\displaystyle c_2< 0$, then from (\ref{increasing of A'}) we know that 
\begin{align}
\mathcal{A}'(t)\leq c_2< 0, \quad \quad \forall t\in \mathbb{R}^+. \nonumber 
\end{align}

Therefore for any $t> 1$, we have 
\begin{align}
0\leq \mathcal{A}(t)= \mathcal{A}(1)+ \int_1^t \mathcal{A}'(s)ds\leq \mathcal{A}(1)+ c_2\cdot (t- 1). \nonumber
\end{align}
Let $t\rightarrow\infty$ in the above, note $c_2< 0$, we get 
\begin{align}
0\leq -\infty, \nonumber 
\end{align}
which is a contradiction. Hence we have $\displaystyle \lim_{t\rightarrow\infty}\mathcal{A}'(t)= 0$.
}
\qed

The following lemma is an analogue of Theorem $2.12$ of \cite{Colding}.
\begin{lemma}\label{lem limit of A-M and V-M}
{For a complete non-compact Riemannian manifold $M^n$, which is non-parabolic with $Rc\geq 0$, we have 
\begin{align}
\lim_{r\rightarrow \infty}\mathcal{A}(r)= n\omega_n\big(\mathrm{V}_M\big)^{\frac{1}{n- 2}} \nonumber 
\end{align}
}
\end{lemma}
\pf
{From Corollary \ref{cor monotonicity of Vt}, we know that $\displaystyle \lim_{t\rightarrow\infty}\mathcal{A}(t)$ exists.

On the other hand, from L'H\^{o}pital's rule,
\begin{align}
\lim\limits_{r\rightarrow \infty}\frac{\int_{b\leq r}|\nabla b|^3}{r^n}= \lim_{r\rightarrow \infty}\frac{\int_{b= r}|\nabla b|^2}{nr^{n- 1}}= \lim_{r\rightarrow \infty}\frac{1}{n}\mathcal{A}(r). \label{limit of A-para}
\end{align}

The conclusion follows from Corollary \ref{cor limit of the rescaled volume by b} and (\ref{limit of A-para}).
}
\qed

\begin{prop}\label{prop integral along b direction-Ricci}
{For a complete non-compact Riemannian manifold $M^n$, which is non-parabolic with $Rc\geq 0$, we have
\begin{align}
\lim_{r\rightarrow \infty}\frac{\int_{b\leq r}|\nabla b|\big[|\Pi_0|^2+ Rc(\vec{n})\big]+ \frac{|\nabla^2 b(\vec{n}, \vec{n})|^2}{(n- 1)|\nabla b|}}{r^{n- 2}}= 0. \nonumber 
\end{align}
}
\end{prop}

\pf
{Taking the integral of (\ref{non-negative 2nd integrand}) with respect to $t$ from $0$ to $r$, using Corollary \ref{cor monotonicity of Vt} we obtain
\begin{align}
&\frac{\int_{b\leq r} \Big\{\operatorname{Ric}(\vec{n}, \vec{n})|\nabla b|+ |\Pi_0|^2\cdot |\nabla b|+ \frac{|\nabla^2 b(\vec{n}, \vec{n})|^2}{(n- 1)|\nabla b|}\Big\}}{r^{n- 2}} \leq \frac{\int_{b\leq r} \mathrm{div}(\frac{\nabla|\nabla b|}{b^{2n- 4}})\cdot b^{2n- 4}}{r^{n- 2}} \nonumber\\
&=\frac{\int_{b\leq r}\Delta |\nabla b|- (2n-4)\frac{\nabla^2b(\vec{n}, \vec{n})}{b}\cdot |\nabla b|}{r^{n- 2}}\nonumber \\
&= r^{2-n}\int_{b= r}\nabla_{\vec{n}}|\nabla b|- (2n-4)\frac{\int_0^r dt\int_{b= t}\frac{\nabla^2b(\vec{n}, \vec{n})}{b}}{r^{n- 2}}\nonumber \\
&= r\mathcal{A}'(r)- (2n-4)\frac{\int_{0}^r\mathcal{A}'(t)t^{n- 2}dt}{r^{n- 2}} \nonumber \\
&\leq - (2n-4)\frac{\int_{0}^r\mathcal{A}'(t)t^{n- 2}dt}{r^{n- 2}}.\nonumber 
\end{align}
Using integration by parts, we get
\begin{align}
    &\frac{\int_{b\leq r} \Big\{\operatorname{Ric}(\vec{n}, \vec{n})|\nabla b|+ |\Pi_0|^2\cdot |\nabla b|+ \frac{|\nabla^2 b(\vec{n}, \vec{n})|^2}{(n- 1)|\nabla b|}\Big\}}{r^{n- 2}}\leq - (2n-4)\frac{\int_{0}^r\mathcal{A}'(t)t^{n- 2}dt}{r^{n- 2}}\nonumber\\
&=-(2n-4)\frac{r^{n- 2}\mathcal{A}(r)- (n- 2)\int_0^r \mathcal{A}(t) t^{n- 3}dt}{r^{n-2}}\nonumber \\
&= (2n-4)\left[-\mathcal{A}(r)+ (n- 2)\frac{\int_0^r \mathcal{A}(t) t^{n- 3}dt}{r^{n-2}}\right]. \label{equation without limit}
\end{align}

Taking the upper limit of (\ref{equation without limit}), using L'H\^{o}pital's rule and Lemma \ref{lem limit of A-M and V-M}, we have
\begin{align}
&\varlimsup_{r\rightarrow\infty}\frac{\int_{b\leq r}|\nabla b|\big[|\Pi_0|^2+ Rc(\vec{n})\big]+ \frac{|\nabla^2 b(\vec{n}, \vec{n})|^2}{(n- 1)|\nabla b|}}{r^{n- 2}}\leq (2n-4)\Big\{-\lim_{r\rightarrow\infty}\mathcal{A}(r)+ \lim_{r\rightarrow\infty}\mathcal{A}(r)\Big\} = 0. \nonumber 
\end{align}

}
\qed



\begin{prop}\label{prop integral along b direction-H2}
{For a complete non-compact Riemannian manifold $M^n$, which is non-parabolic with $Rc\geq 0$, we have
\begin{align}
&\lim_{r\rightarrow \infty}r^{2- n}\int_{b\leq r}|\nabla b|\cdot H^2= \frac{(n- 1)^2n}{n- 2}\omega_n\big(\mathrm{V}_M\big)^{\frac{1}{n- 2}} , \nonumber 
\end{align}
where $H$ is the mean curvature of the level set of $b$ with respect to the normal vector $\frac{\nabla b}{|\nabla b|}$.
}
\end{prop}

\pf
{Using $\Delta b= (n-1)\frac{|\nabla b|^2}{b}$, from (\ref{mean curv formula}), Proposition \ref{prop integral along b direction-Ricci} and L'H\^{o}pital's rule, we have 
\begin{align}
&\lim_{r\rightarrow \infty}r^{2- n}\int_{b\leq r}|\nabla b|\cdot H^2 \nonumber \\
&= \lim_{r\rightarrow \infty}r^{2- n}\int_{b\leq r}|\nabla b|\cdot \Big(\frac{\Delta b}{|\nabla b|}- \frac{\nabla^2 b(\vec{n}, \vec{n})}{|\nabla b|}\Big)^2 \nonumber \\
&= \lim_{r\rightarrow \infty}\frac{\int_{b\leq r}|\nabla b|\cdot \Big(\frac{\Delta b}{|\nabla b|}\Big)^2}{r^{n- 2}}+ \frac{\int_{b\leq r}|\nabla b|\cdot \Big(\frac{\nabla^2 b(\vec{n}, \vec{n})}{|\nabla b|}\Big)^2}{r^{n- 2}}- 2 \frac{\int_{b\leq r}\Delta b\frac{\nabla^2 b(\vec{n}, \vec{n})}{|\nabla b|}}{r^{n- 2}} \nonumber \\
&= \lim_{r\rightarrow \infty}\frac{(n- 1)^2\int_{b\leq r}\frac{|\nabla b|^3}{b^2}}{r^{n- 2}}+ \frac{\int_{b\leq r}\frac{|\nabla^2 b(\vec{n}, \vec{n})|^2}{|\nabla b|}}{r^{n- 2}}- 2(n- 1) \frac{\int_{b\leq r}|\nabla b|\frac{\nabla^2 b(\vec{n}, \vec{n})}{b}}{r^{n- 2}} \nonumber \\
&= \lim_{r\rightarrow \infty}\frac{(n- 1)^2\int_{b= r}\frac{|\nabla b|^2}{b^2}}{(n- 2)r^{n- 3}}- 2(n- 1) \frac{\int_{b\leq r}|\nabla b|\frac{\nabla^2 b(\vec{n}, \vec{n})}{b}}{r^{n- 2}} \nonumber \\
&= \lim_{r\rightarrow \infty}\frac{(n- 1)^2}{n- 2}\mathcal{A}(r)- 2(n- 1) \frac{\int_{b\leq r}|\nabla b|\frac{\nabla^2 b(\vec{n}, \vec{n})}{b}}{r^{n- 2}} .\label{limit with two terms}
\end{align}

Note $\mathcal{A}'(t)= t^{1- n}\int_{b= t}\nabla^2 b(\vec{n}, \vec{n})$, integration by parts yields 
\begin{align}
&\lim_{r\rightarrow \infty}\frac{\int_{b\leq r}|\nabla b|\frac{\nabla^2 b(\vec{n}, \vec{n})}{b}}{r^{n- 2}}= \lim_{r\rightarrow \infty}\frac{\int_0^r dt\int_{b= t}\frac{\nabla^2 b(\vec{n}, \vec{n})}{b}}{r^{n- 2}} = \lim_{r\rightarrow \infty}\frac{\int_0^r t^{n- 2}\mathcal{A}'(t) dt}{r^{n- 2}} \nonumber \\
&= \lim_{r\rightarrow \infty}\frac{\mathcal{A}(r)r^{n- 2}- (n- 2)\int_0^r t^{n- 3}\mathcal{A}(t) dt}{r^{n- 2}} \nonumber \\
&= \lim_{r\rightarrow \infty}\mathcal{A}(r)- (n- 2)\frac{\int_0^r t^{n- 3}\mathcal{A}(t) dt}{r^{n- 2}}. \label{limit of the subtle term} 
\end{align}

From Lemma \ref{lem limit of A-M and V-M} and (\ref{limit of the subtle term}), using L'H\^{o}pital's rule, we get
\begin{align}
&\lim_{r\rightarrow \infty}\frac{\int_{b\leq r}|\nabla b|\frac{\nabla^2 b(\vec{n}, \vec{n})}{b}}{r^{n- 2}}= \lim_{r\rightarrow \infty}\mathcal{A}(r)- \lim_{r\rightarrow \infty}\mathcal{A}(r)= 0. \label{the subtle limit is 0}
\end{align}

Now from Lemma \ref{lem limit of A-M and V-M}, (\ref{limit with two terms}) and (\ref{the subtle limit is 0}), we  get
\begin{align}
&\lim_{r\rightarrow \infty}r^{2- n}\int_{b\leq r}|\nabla b|\cdot H^2= \frac{(n- 1)^2n}{n- 2}\omega_n\big(\mathrm{V}_M\big)^{\frac{1}{n- 2}} . \nonumber 
\end{align}
}
\qed

\section{Integral formula for scalar curvature}

\begin{lemma}\label{lem measure and integral ineq}
{For a complete non-compact, non-parabolic Riemannian manifold $M^n$ with $Rc\geq 0$, for any $A\subseteq \mathbb{R}^+- \mathcal{S}$, we have 
\begin{align}
&\big(\mathrm{V}_M\big)^{\frac{1}{n- 2}}\cdot \varlimsup_{r\rightarrow\infty}\frac{\int_{A_r}t^{n- 3}dt}{r^{n- 2}} \leq \varlimsup_{r\rightarrow\infty}\frac{\int_{b^{-1}(A_r)}R(b^{-1}(t))\cdot |\nabla b|}{(n- 2)(n- 1)n\omega_n r^{n- 2}}, \label{meas sup}\\
&\big(\mathrm{V}_M\big)^{\frac{1}{n- 2}}\cdot \varliminf_{r\rightarrow \infty}\frac{\int_{A_r}t^{n- 3}dt}{r^{n- 2}} \leq \varliminf_{r\rightarrow\infty}\frac{\int_{b^{-1}(A_r)}R(b^{-1}(t))\cdot |\nabla b|}{(n- 2)(n- 1)n\omega_n r^{n- 2}}, \label{meas inf}
\end{align}
where $A_r\vcentcolon= A\cap (0, r]$ for any $r> 0$. 
}
\end{lemma}

\pf
{\textbf{Step (1)}. From the Gauss equation, we have 
\begin{align}
R(M^n)&= R\big(b^{-1}(t)\big)+ 2Rc(\vec{n})- 2\sum_{i\neq j}\lambda_i\lambda_j \nonumber \\
&= R\big(b^{-1}(t)\big)+ 2Rc(\vec{n})- \Big(\frac{n- 2}{n- 1}H^2- |\Pi_0|^2\Big) \nonumber \\
&= R\big(b^{-1}(t)\big)+ 2Rc(\vec{n})+ |\Pi_0|^2- \frac{n- 2}{n- 1}H^2. \label{scalar curv in term of Ricci and H-n-dim} 
\end{align}

From (\ref{scalar curv in term of Ricci and H-n-dim}) and $R(M^n)\geq 0$, we get
\begin{align}
&\frac{\int_{b^{-1}(A_r)}R(b^{-1}(t))\cdot |\nabla b|}{r^{n- 2}}\geq \frac{n- 2}{n- 1} \frac{\int_{b^{-1}(A_r)}H^2\cdot |\nabla b|}{r^{n- 2}}- \frac{\int_{b^{-1}(A_r)}[2Rc(\vec{n})+ |\Pi_0|^2]\cdot |\nabla b|}{r^{n- 2}}. \label{ineq of curv and second fund} 
\end{align}

Taking the limit of (\ref{ineq of curv and second fund}), from Proposition \ref{prop integral along b direction-Ricci}, we have
\begin{align}
&\varlimsup_{r\rightarrow\infty}\frac{\int_{b^{-1}(A_r)}R(b^{-1}(t))\cdot |\nabla b|}{r^{n- 2}}\geq \frac{n- 2}{n- 1} \varlimsup_{r\rightarrow\infty}\frac{\int_{b^{-1}(A_r)}H^2\cdot |\nabla b|}{r^{n- 2}}\label{ineq with R and H2} \\
&\varliminf_{r\rightarrow\infty}\frac{\int_{b^{-1}(A_r)}R(b^{-1}(t))\cdot |\nabla b|}{r^{n- 2}}\geq \frac{n- 2}{n- 1} \varliminf_{r\rightarrow\infty}\frac{\int_{b^{-1}(A_r)}H^2\cdot |\nabla b|}{r^{n- 2}}. \label{ineq with R and H2-inf}
\end{align}

\textbf{Step (2)}. From Lemma \ref{lem mean curv of level set in form of bp}, we know 
\begin{align}
H= \frac{(n- 1)|\nabla b|}{b}- \frac{\nabla^2 b(\vec{n}, \vec{n})}{|\nabla b|} .\label{expression of H in form of b}
\end{align}

From (\ref{expression of H in form of b}) and Proposition \ref{prop integral along b direction-Ricci}, we get
\begin{align}
&\varlimsup_{r\rightarrow \infty}\frac{\int_{b^{-1}(A_r)}|\nabla b|\cdot H^2}{r^{n- 2}}\nonumber\\
&= \varlimsup_{r\rightarrow \infty}\frac{\int_{b^{-1}(A_r)} \Big\{(n- 1)^2\frac{|\nabla b|^3}{b^2}- 2(n- 1)\frac{|\nabla b|\cdot \nabla^2 b(\vec{n}, \vec{n})}{b} + \frac{|\nabla^2 b(\vec{n}, \vec{n})|^2}{|\nabla b|}\Big\}}{r^{n- 2}} \nonumber \\
&\geq ((n- 1)^2- \epsilon)\varlimsup_{r\rightarrow \infty}\frac{\int_{b^{-1}(A_r)} \frac{|\nabla b|^3}{b^2}}{r^{n- 2}}+ \big(1- \frac{(n- 1)^2}{\epsilon}\big)\lim_{r\rightarrow \infty}
\frac{\int_{b^{-1}(A_r)} \frac{|\nabla^2 b(\vec{n}, \vec{n})|^2}{|\nabla b|}}{r^{n- 2}} \nonumber \\
&\geq ((n- 1)^2- \epsilon)\varlimsup_{r\rightarrow \infty}\frac{\int_{b^{-1}(A_r)} \frac{|\nabla b|^3}{b^2}}{r^{n- 2}}. \nonumber 
\end{align}

Let $\epsilon\rightarrow 0$ in the above, we have 
\begin{align}
\varlimsup_{r\rightarrow \infty}\frac{\int_{b^{-1}(A_r)}|\nabla b|\cdot H^2}{r^{n- 2}}\geq (n- 1)^2\varlimsup_{r\rightarrow \infty}\frac{\int_{b^{-1}(A_r)} \frac{|\nabla b|^3}{b^2}}{r^{n- 2}} . \nonumber 
\end{align}

By Lemma \ref{lem limit of A-M and V-M} and the above, we have 
\begin{align}
&\varlimsup_{r\rightarrow \infty}\frac{\int_{b^{-1}(A_r)}|\nabla b|\cdot H^2}{r^{n- 2}}\geq (n- 1)^2\varlimsup_{r\rightarrow \infty}\frac{\int_{A_r}dt\int_{b^{-1}(t)} \frac{|\nabla b|^2}{b^2}}{r^{n- 2}} \nonumber \\
&\geq (n- 1)^2\varlimsup_{r\rightarrow \infty}\frac{\int_{A_r}t^{n- 3}dt}{r^{n- 2}}\cdot \lim_{r\rightarrow\infty} \mathcal{A}(r)\\
&= (n-1)^2n\omega_n (\mathrm{V}_{M})^{\frac{1}{n- 2}}\cdot\varlimsup_{r\rightarrow \infty}\frac{\int_{A_r}t^{n- 3}dt}{r^{n- 2}}. \label{lower bound of H2}
\end{align}

Similarly, we can get
\begin{align}
&\varliminf_{r\rightarrow \infty}\frac{\int_{b^{-1}(A_r)}|\nabla b|\cdot H^2}{r^{n- 2}}\geq (n-1)^2n\omega_n (\mathrm{V}_{M})^{\frac{1}{n- 2}}\cdot \varliminf_{r\rightarrow \infty}\frac{\int_{A_r}t^{n- 3}dt}{r^{n- 2}}. \label{lower bound of H2-inf}
\end{align}

From (\ref{ineq with R and H2}) and (\ref{lower bound of H2}), we obtain (\ref{meas sup}).

Similarly from (\ref{ineq with R and H2-inf}) and (\ref{lower bound of H2-inf}), we can obtain (\ref{meas inf}).
}
\qed

\begin{theorem}\label{thm 3-mflds with non-negative Ricci}
{For a complete non-compact $3$-dim Riemannian manifold $(M^3, g)$ with $Rc\geq 0$, then one of the following holds:
\begin{enumerate}
\item[(a).] $M^3$ is diffeomorphic to $\mathbb{R}^3$;
\item[(b).] the universal cover of $M^3$ is isometric to $N^2\times \mathbb{R}$, where $N^2$ is a complete Riemannian manifolds with sectional curvature $K\geq 0$.
\end{enumerate}
}
\end{theorem}

\pf
{See \cite[Theorem $2$]{Liu} (also see \cite{CMS} for an alternative proof).
}
\qed

\begin{prop}\label{prop topology of non-parabolic} 
{For a complete non-compact Riemannian manifold $(M^3, g)$ with $Rc\geq 0$, if $(M^3, g)$ is non-parabolic, then $M^3$ is diffeomorphic to $\mathbb{R}^3$.
}
\end{prop}

\pf
{\textbf{Step (1)}. By contradiction, assume $M^3$ is not diffeomorphic to $\mathbb{R}^3$. 

Then we consider the universal cover $\tilde{M}$ of $M^3$. By Theorem \ref{thm 3-mflds with non-negative Ricci}, we get that $\tilde{M}$ is isometric to $N^2\times \mathbb{R}$. 

Because $M^3$ is non-parabolic, $\tilde{M}= N^2\times \mathbb{R}$ is also non-parabolic. By $\pi_1(\tilde{M})= 0$, we know that $N^2$ is a complete non-compact surface with the sectional curvature $K(N^2)\geq 0$ and $\pi_1(N^2)= 0$. 

By Cheeger-Gromoll's soul theorem and $\pi_1(N^2)= 0$, we know that $N^2$ is diffeomorphic to $\mathbb{R}^2$. Hence $\tilde{M}$ is diffeomorphic to $\mathbb{R}^3$.  




\textbf{Step (2)}. Let $\Gamma = \pi_1(M)$ be the deck transformation group acting freely and 
isometrically on $\tilde{M} = N^2 \times \mathbb{R}$. 
Since $M$ is not diffeomorphic to $\mathbb{R}^3$, $\Gamma$ is nontrivial. 


Each $\gamma \in \Gamma$ decomposes as $\gamma = (\gamma_1, \gamma_2)$, where 
$\gamma_1 \in \mathrm{Isom}(N^2)$ and $\gamma_2 \in \mathrm{Isom}(\mathbb{R})$. Define
\begin{align}
p: \Gamma\rightarrow \mathrm{Isom}(\mathbb{R}), \quad \quad p(\gamma_1, \gamma_2)= \gamma_2, \quad \quad \forall (\gamma_1, \gamma_2)\in \Gamma. \nonumber 
\end{align}

If $p(\Gamma)$ is trivial, then $\Gamma$ acts only on $N^2$. Thus $M = (N^2 / \Gamma) \times \mathbb{R}$. By $M$ is non-parabolic, we get $S= N^2/\Gamma$ is a noncompact surface with $K \ge 0$ and non-zero fundamental group. 

From Cheeger-Gromoll's soul theorem, we get that $S$ is diffeomorphic to $S^1 \times \mathbb{R}$. This implies $S$ has two ends. By Cheeger-Gromoll's splitting theorem, we know that $S$ is isometric to a flat cylinder $S^1 \times \mathbb{R}$. Therefore, we have
\begin{align}
\varlimsup_{r \to \infty} \frac{V(B_p(r))}{r^2} < \infty. \nonumber 
\end{align}
It contradicts that $M^3$ is non-parabolic.

\textbf{Step (3)}. In the rest, we assume $p(\Gamma)$ is infinite cyclic generated by a translation on $\mathbb{R}$.

Then we get that $M$ is isometric to $F \times S^1$ (by taking the standard product metric on $F \times \mathbb{R}$ 
invariant under the $\mathbb{Z}$-action, where $F= N/\Gamma_0$ for some $\Gamma_0$). 

Thus, $M$ is isometric to $F \times S^1$ with a product metric, where $F$ is a complete surface with $K \ge 0$. And we get $V(B_p(r))\leq C r^2$. 
 
Hence, $\displaystyle \varlimsup_{r \to \infty} \frac{V(B_p(r))}{r^2} < \infty$, it is a contradiction again.
}
\qed

Now we can remove the topological restriction in \cite[Lemma $3.4$]{Xu-IOC1} as follows.
\begin{lemma}\label{lem 3-dim level set is connected}
{For a complete, non-parabolic Riemannian manifold $M^3$ with $Rc\geq 0$, if $b^{-1}(t)$ is a smooth surface for some $t> 0$, then $b^{-1}(t)$ is connected.
}
\end{lemma}

\pf{From Lemma \ref{lem point-wise bound of b}, we know that $b^{-1}(t)$ is bounded.

If $\Omega_1, \Omega_2$ are two connected components of $b^{-1}(t)$, then the base point of $b$, denoted as $q$, is enclosed by the unique surface $\Omega_i$ (otherwise $G\equiv t^{-1}$ on the region enclosed by $\Omega_1$ and $\Omega_2$ from the maximum principle, from the unique continuation of harmonic function, we get the contradiction). 

Note $M^3$ is diffeomorphic to $\mathbb{R}^3$, hence the other $\Omega_{j}$ encloses one region $\Omega\subset M^3$, and $q\notin \Omega$. By the maximum principle again $G\equiv t^{-1}$ on $\Omega$, the contradiction follows from the unique continuation of harmonic function again.
}
\qed

For any $r> 0$, define 
\begin{align}
&\mathcal{T}\vcentcolon= \{t\in \mathbb{R}^+- \mathcal{S}: b^{-1}(t)\ \text{is diffeomorphic to}\ \mathbb{T}^2\}, \nonumber \\
&\mathcal{B}\vcentcolon= \{t\in \mathbb{R}^+- \mathcal{S}: \text{the genus of}\ b^{-1}(t)\ \text{is greater than}\ 1\}, \nonumber \\
&\mathcal{G}\vcentcolon= \{t\in \mathbb{R}^+- \mathcal{S}: b^{-1}(t)\ \text{is diffeomorphic to}\ \mathbb{S}^2\},\nonumber \\
&\mathcal{T}_r= \mathcal{T}\cap (0, r], \quad \quad \mathcal{B}_r= \mathcal{B}\cap (0, r], \quad \quad \mathcal{G}_r= \mathcal{G}\cap (0, r]. \nonumber 
\end{align}

\begin{lemma}\label{lem bad values has density 0}
{For a complete, non-parabolic Riemannian manifold $M^3$ with $Rc\geq 0$, we have 
\begin{align}
&\varliminf_{r\rightarrow\infty}\frac{\mathcal{L}^1(\mathcal{B}_r)}{r}= 0. \label{lower limit for higher genus is 0}\\
&\lim_{r\rightarrow\infty}\frac{\int_{\mathcal{B}_r\cup \mathcal{T}_r}dt\int_{b^{-1}(t)}R(b^{-1}(t))}{r}= 0. \label{limit of integral of bad curv}
\end{align}
And if $M^3$ has maximal volume growth, then 
\begin{align}
&\lim_{r\rightarrow\infty}\frac{\mathcal{L}^1(\mathcal{T}_r\cup \mathcal{B}_r)}{r}= 0. \label{limit of positive genus is 0}
\end{align}
}
\end{lemma}

\pf
{\textbf{Step (1)}. From Lemma \ref{lem measure and integral ineq} and Gauss-Bonnet's Theorem, we have 
\begin{align}
&0\leq \mathrm{V}_M\cdot \varlimsup_{r\rightarrow\infty}\frac{\mathcal{L}^1(\mathcal{B}_r)}{r} \leq \varlimsup_{r\rightarrow\infty}\frac{\int_{b^{-1}(\mathcal{B}_r)}R(b^{-1}(t))\cdot |\nabla b|}{8\pi r}, \nonumber\\
&= \varlimsup_{r\rightarrow\infty}\frac{\int_{\mathcal{B}_r} dt\int_{b^{-1}(t)}R(b^{-1}(t))}{8\pi r}\leq -\varliminf_{r\rightarrow\infty}\frac{\mathcal{L}^1(\mathcal{B}_r)}{r}\leq 0. \nonumber
\end{align}
Then we get  (\ref{lower limit for higher genus is 0}).

From Gauss-Bonnet's Theorem, we know $\int_{b^{-1}(t)}R(b^{-1}(t))\leq 0$ for any $t\in \mathcal{T}_r\cup \mathcal{B}_r$, then
\begin{align}
\varlimsup_{r\rightarrow\infty}\frac{\int_{\mathcal{T}_r\cup \mathcal{B}_r} dt\int_{b^{-1}(t)}R(b^{-1}(t))}{8\pi r}\leq  0. \label{limsup is no more than 0}
\end{align}

On the other hand, From Lemma \ref{lem measure and integral ineq} we have 
\begin{align}
&0\leq  \mathrm{V}_M\cdot \varliminf_{r\rightarrow\infty}\frac{\mathcal{L}^1(\mathcal{T}_r\cup \mathcal{B}_r)}{r} \leq \varliminf_{r\rightarrow\infty}\frac{\int_{b^{-1}(\mathcal{T}_r\cup \mathcal{B}_r)}R(b^{-1}(t))\cdot |\nabla b|}{8\pi r}, \nonumber\\
&=\varliminf_{r\rightarrow\infty}\frac{\int_{\mathcal{T}_r\cup \mathcal{B}_r} dt\int_{b^{-1}(t)}R(b^{-1}(t))}{8\pi r}. \label{liminf is 0}
\end{align}

Combining (\ref{liminf is 0}) and (\ref{limsup is no more than 0}), we get (\ref{limit of integral of bad curv}). 

\textbf{Step (2)}. By Lemma \ref{lem measure and integral ineq} and (\ref{limsup is no more than 0}), we get  
\begin{align}
&0\leq \mathrm{V}_M\cdot \varlimsup_{r\rightarrow\infty}\frac{\mathcal{L}^1(\mathcal{T}_r\cup \mathcal{B}_r)}{r} \leq \varlimsup_{r\rightarrow\infty}\frac{\int_{b^{-1}(\mathcal{T}_r\cup \mathcal{B}_r)}R(b^{-1}(t))\cdot |\nabla b|}{8\pi r}, \nonumber\\
&= \varlimsup_{r\rightarrow\infty}\frac{\int_{\mathcal{T}_r\cup \mathcal{B}_r} dt\int_{b^{-1}(t)}R(b^{-1}(t))}{8\pi r}\leq  0. \nonumber 
\end{align}
This implies $\displaystyle \varlimsup_{r\rightarrow\infty}\frac{\mathcal{L}^1(\mathcal{T}_r\cup \mathcal{B}_r)}{r}= 0$ by $\mathrm{V}_M> 0$, hence (\ref{limit of positive genus is 0}) follows.
}
\qed

The main result of this section is as follows.
\begin{theorem}\label{thm equa of integ curv}
{For a complete, non-parabolic Riemannian manifold $M^3$ with $Rc\geq 0$ and maximal volume growth, there is
\begin{align}
\lim_{r\rightarrow \infty}\frac{\int_{b\leq r}R\cdot |\nabla b|}{r}=8\pi \big[1- \mathrm{V}_{M}\big].\nonumber 
\end{align}
}
\end{theorem}

\pf
{From Proposition \ref{prop topology of non-parabolic} we know that $M^3$ is diffeomorphic to $\mathbb{R}^3$. Now using Lemma \ref{lem 3-dim level set is connected}, we know each smooth surface $b^{-1}(t)$ is connected.

From the definition of $\mathcal{G}_r$ and Gauss-Bonnet's Theorem, we obtain
\begin{align}
\int_{b^{-1}(t)}R(b^{-1}(t))= 8\pi, \quad \quad \quad \forall t\in \mathcal{G}_r. \label{integral on surface}
\end{align}

Now from Lemma \ref{lem bad values has density 0} and (\ref{integral on surface}), we have 
\begin{align}
&\quad \lim_{r\rightarrow\infty}\frac{\int_{(0, r]- \mathcal{S}}dt\int_{b^{-1}(t)}R(b^{-1}(t))}{r} \nonumber \\
&= \lim_{r\rightarrow\infty}\frac{\int_{\mathcal{G}_r}dt\int_{b^{-1}(t)}R(b^{-1}(t))}{r}+ \lim_{r\rightarrow\infty}\frac{\int_{\mathcal{B}_r\cup \mathcal{T}_r}dt\int_{b^{-1}(t)}R(b^{-1}(t))}{r} \nonumber \\
&= 8\pi\cdot \lim_{r\rightarrow\infty}\frac{\mathcal{L}^1(\mathcal{G}_r)}{r}+ 0= 8\pi. \label{sphere integral key}
\end{align}

From the Gauss equation, we have 
\begin{align}
R(M^n)&= R\big(b^{-1}(t)\big)+ 2Rc(\vec{n})+ |\Pi_0|^2- \frac{1}{2}H^2. \label{scalar curv in term of Ricci and H} 
\end{align}

From Proposition \ref{prop integral along b direction-Ricci} and Proposition \ref{prop integral along b direction-H2}, using (\ref{scalar curv in term of Ricci and H}) and (\ref{sphere integral key}) we obtain
\begin{align}
\lim_{r\rightarrow \infty}\frac{\int_{b\leq r}R\cdot |\nabla b|}{r}= \lim_{r\rightarrow\infty}\frac{\int_0^rdt\int_{b^{-1}(t)}R(b^{-1}(t))}{r} - 8\pi \mathrm{V}_{M}= 8\pi \big[1- \mathrm{V}_{M}\big]. \nonumber 
\end{align}
}
\qed

\begin{cor}\label{cor Hamilton-Pinch conj}
{Assume $(M^3, g)$ is a complete Riemannian manifold with $Rc\geq \epsilon\cdot R\cdot g\geq 0$, where $\epsilon> 0$ is a fixed constant. Furthermore assume $(M^3, g)$ is maximal volume growth, then $M^3$ is isometric to $\mathbb{R}^3$.
}
\end{cor}

\pf
{From Proposition \ref{prop integral along b direction-Ricci} and Theorem \ref{thm equa of integ curv}, we obtain
\begin{align}
0= \lim_{r\rightarrow \infty}\frac{1}{r}\int_{b\leq r}Rc(\vec{n})\cdot |\nabla b|\geq \epsilon\cdot \lim_{r\rightarrow \infty}\frac{1}{r}\int_{b\leq r}R\cdot |\nabla b|= \epsilon\cdot 8\pi(1- \mathrm{V}_{M^3}). \nonumber 
\end{align}
Hence we get $\mathrm{V}_{M^3}= 1$, which implies $M^3$ is isometric to $\mathbb{R}^3$ by the rigidity part of Bishop-Gromov's volume comparison Theorem.






}
\qed

\end{document}